\documentclass[a4paper,10pt]{amsart}
\usepackage{amsmath,amsthm,amssymb,latexsym,enumerate,color,hyperref}
\usepackage{graphicx}

\usepackage{xcolor}

\usepackage{comment}
\usepackage{eucal}

\usepackage{tikz}

\numberwithin{equation}{section}

\begin{document}

\newtheorem{thm}{Theorem}[section]
\newtheorem{prop}[thm]{Proposition}
\newtheorem{lem}[thm]{Lemma}
\newtheorem{cor}[thm]{Corollary}
\newtheorem{rem}[thm]{Remark}
\newtheorem{defn}[thm]{Definition}

\newcommand{\DD}{\mathbb{D}}
\newcommand{\NN}{\mathbb{N}}
\newcommand{\ZZ}{\mathbb{Z}}
\newcommand{\QQ}{\mathbb{Q}}
\newcommand{\RR}{\mathbb{R}}
\newcommand{\CC}{\mathbb{C}}
\renewcommand{\SS}{\mathbb{S}}

\renewcommand{\theequation}{\arabic{section}.\arabic{equation}}

\newcommand{\supp}{\mathop{\mathrm{supp}}}    

\newcommand{\re}{\mathop{\mathrm{Re}}}   
\newcommand{\im}{\mathop{\mathrm{Im}}}   
\newcommand{\dist}{\mathop{\mathrm{dist}}}  
\newcommand{\link}{\mathop{\circ\kern-.35em -}}
\newcommand{\spn}{\mathop{\mathrm{span}}}   
\newcommand{\ind}{\mathop{\mathrm{ind}}}   
\newcommand{\rank}{\mathop{\mathrm{rank}}}   
\newcommand{\Fix}{\mathop{\mathrm{Fix}}}   
\newcommand{\codim}{\mathop{\mathrm{codim}}}   
\newcommand{\conv}{\mathop{\mathrm{conv}}}   
\newcommand{\epsi}{\mbox{$\varepsilon$}}
\newcommand{\eps}{\mathchoice{\epsi}{\epsi}
{\mbox{\scriptsize\epsi}}{\mbox{\tiny\epsi}}}
\newcommand{\cl}{\overline}
\newcommand{\pa}{\partial}
\newcommand{\ve}{\varepsilon}
\newcommand{\zi}{\zeta}
\newcommand{\Si}{\Sigma}
\newcommand{\cA}{{\mathcal A}}
\newcommand{\cG}{{\mathcal G}}
\newcommand{\cH}{{\mathcal H}}
\newcommand{\cI}{{\mathcal I}}
\newcommand{\cJ}{{\mathcal J}}
\newcommand{\cK}{{\mathcal K}}
\newcommand{\cL}{{\mathcal L}}
\newcommand{\cN}{{\mathcal N}}
\newcommand{\cR}{{\mathcal R}}
\newcommand{\cS}{{\mathcal S}}
\newcommand{\cT}{{\mathcal T}}
\newcommand{\cU}{{\mathcal U}}
\newcommand{\OM}{\Omega}
\newcommand{\B}{\bullet}
\newcommand{\ol}{\overline}
\newcommand{\ul}{\underline}
\newcommand{\vp}{\varphi}
\newcommand{\AC}{\mathop{\mathrm{AC}}}   
\newcommand{\Lip}{\mathop{\mathrm{Lip}}}   
\newcommand{\es}{\mathop{\mathrm{esssup}}}   
\newcommand{\les}{\mathop{\mathrm{les}}}   
\newcommand{\nid}{\noindent}
\newcommand{\pzr}{\phi^0_R}
\newcommand{\pir}{\phi^\infty_R}
\newcommand{\psr}{\phi^*_R}
\newcommand{\pow}{\frac{N}{N-1}}
\newcommand{\ncl}{\mathop{\mathrm{nc-lim}}}   
\newcommand{\nvl}{\mathop{\mathrm{nv-lim}}}  
\newcommand{\la}{\lambda}
\newcommand{\La}{\Lambda}    
\newcommand{\de}{\delta}    
\newcommand{\fhi}{\varphi} 
\newcommand{\ga}{\gamma}    
\newcommand{\ka}{\kappa}   

\newcommand{\core}{\heartsuit}
\newcommand{\diam}{\mathrm{diam}}

\newcommand{\lan}{\langle}
\newcommand{\ran}{\rangle}
\newcommand{\tr}{\mathop{\mathrm{tr}}}
\newcommand{\diag}{\mathop{\mathrm{diag}}}
\newcommand{\dv}{\mathop{\mathrm{div}}}

\newcommand{\al}{\alpha}
\newcommand{\be}{\beta}
\newcommand{\Om}{\Omega}
\newcommand{\na}{\nabla}

\newcommand{\cC}{\mathcal{C}}
\newcommand{\cM}{\mathcal{M}}
\newcommand{\nr}{\Vert}
\newcommand{\De}{\Delta}
\newcommand{\cX}{\mathcal{X}}
\newcommand{\cP}{\mathcal{P}}
\newcommand{\om}{\omega}
\newcommand{\si}{\sigma}
\newcommand{\te}{\theta}
\newcommand{\Ga}{\Gamma}

\newcommand{\loc}{\mathrm{loc}}

\newcommand{\ca}{\tilde{a}}


\title[Quantitative symmetry in a mixed 
Serrin-type problem]{Quantitative symmetry \\ in a mixed  
Serrin-type problem \\ for a constrained torsional rigidity}

\author{Rolando Magnanini} 
\address{Dipartimento di Matematica ed Informatica ``U.~Dini'',
Universit\` a di Firenze, viale Morgagni 67/A, 50134 Firenze, Italy.}
    \email{rolando.magnanini@unifi.it}
    \urladdr{http://web.math.unifi.it/users/magnanin}

\author{Giorgio Poggesi}
\address{Department of Mathematics and Statistics, The University of Western Australia, 35 Stirling Highway, Crawley, Perth, WA 6009, Australia}
    \email{giorgio.poggesi@uwa.edu.au}


\begin{abstract}
We consider a mixed boundary value problem in a domain $\Om$ contained in a half-ball $B_+$ and having a portion $\ol{T}$ of its boundary in common with the curved part of $\pa B_+$. The problem has to do with some sort of constrained torsional rigidity. In this situation, the relevant solution $u$ satisfies a Steklov condition on $T$ and a homogeneous Dirichlet condition on $\Si=\pa\Om\setminus \ol{T} \subset B_+$.  We provide an integral identity that relates (a symmetric function of) the second derivatives of the solution in $\Om$ to  its normal derivative $u_\nu$ on $\Si$. A first significant consequence of this identity is a rigidity result under a quite weak overdetermining integral condition for $u_\nu$ on $\Si$: in fact, it turns out that $\Si$ must be a spherical cap that meets $T$ orthogonally. This result returns the one obtained by J. Guo and C. Xia under the stronger pointwise condition that the values of $u_\nu$ be constant on  $\Si$.  A second important consequence is a set of stability bounds, which quantitatively measure how $\Si$ is far uniformly from being a spherical cap, if $u_\nu$ deviates from a constant in the norm $L^1(\Si)$.
\end{abstract}


\maketitle

\raggedbottom

\section{Introduction}

Let $B$ and $S=\pa B$ be the (open) unit ball and the unit sphere in $\RR^N$, centered at the origin, and set $B_+=\{ x=(x_1,\dots, x_N)\in B : x_N>0\}$. Consider in $B_+$
a bounded domain $\Om$ (i.e., a bounded open connected set) whose boundary $\Ga$
is the union of $\ol{\Si}$ and $T= \Ga \setminus \ol{\Si}$, where $\Si$ is a smooth hypersurface contained in $B_+$, $T$ is a subset of $S$, and $\ol{T}$ meets $\ol{\Si}$ at a common $(N-2)$-dimensional submanifold $\La= \ol{\Si} \cap \ol{T}$ of $S$.
\par
In $\Om$, we consider the following mixed
boundary value problem:
\begin{equation}
\label{eq:problem}
\De u = N \ \text{  in  } \ \Om, \quad u=0 \ \text{  on  } \ \Si , \quad u_{\nu} = u \ \text{  on  } \ T.
\end{equation}
Here and in what follows,  $u_\nu$ denotes the derivative of $u$ in the direction of  the outward unit normal $\nu$ to $\Ga$. In particular, we have that
\begin{equation*}
\nu(x)=x \ \text{ on } \ T.
\end{equation*}

For the existence and uniqueness of a solution 
$u$ of \eqref{eq:problem}, which is smooth in $\ol{\Om} \setminus \La$ and belongs to $C^{0, \ga} (\ol{\Om})$ for some $\ga \in (0,1)$, we refer the reader to \cite[Proposition~2.2]{GX}; more precisely, \cite{Lieberman} guarantees that $u \in C^{0, \ga} (\ol{\Om})$, whereas the classical regularity theory for elliptic equations gives that, for $k\in\NN$ with $k \ge 2$ and $\ga \in (0,1)$, $u \in C^{k,\ga} ( \ol{\Om} \setminus \La)$ provided that $\Si$ is of class $C^{k,\ga}$.
In \cite{GX}, the solution of \eqref{eq:problem} is obtained as a suitably normalized solution of the variational problem:
$$
\sup_{0\ne v\in W^{1,2}_0(\Om, \Si)}\frac{\left(\int_\Om v\,dx\right)^2}{\int_\Om |\na v|^2 dx-\int_T v^2 dS_x}.
$$
Here, $W^{1,2}_0(\Om, \Si)$ denotes the subspace of functions in $W^{1,2} (\Om)$ vanishing 
in a Sobolev sense on $\Si$. The supremum can be interpreted as some sort of \textit{relative} or \textit{constrained torsion} $\cT(\Om,\Si)$ for domains contained in $B$. In fact, if $\ol{\Om}\subset B$ (and hence $T=\varnothing$), we recover a definition of the standard torsion of $\Om$ (see \cite{Ma}).


The issue of regularity up to the (whole) boundary for \eqref{eq:problem} is delicate. The regularity of the solution $u$ strongly depends on how $\ol{\Si}$ and $\ol{T}$ intersect. 
As done in \cite{GX}, we shall further assume that $u$ belongs to $W^{1,\infty}(\Om) \cap W^{2,2} (\Om)$ to ensure that we can integrate by parts.
As shown in \cite[Proposition 3.5]{GX}, such an assumption is surely satisfied whenever $\ol{\Si}$ and $\ol{T}$ intersect orthogonally.

%
%
%

The aim of this paper is to study a Serrin-type overdetermined boundary value problem for \eqref{eq:problem}. In fact, similarly to \cite{Se}, \cite{We} and as done in \cite{GX}, we add the extra condition
\begin{equation}
\label{extra-condition}
u_\nu=R \ \mbox{ on } \ \Si,
\end{equation}
where $R$ is some given constant. In \cite{GX}, under suitable regularity assumptions, it is shown that the problem \eqref{eq:problem}-\eqref{extra-condition} arises naturally in a shape optimization problem. If the relative torsion $\cT(\Om,\Si)$ is stationary with respect to volume-preserving transformations at a domain $\Om$, then the corresponding function $u$ that attains $\cT(\Om,\Si)$ satisfies \eqref{eq:problem}-\eqref{extra-condition} (see \cite[Proposition 4.2]{GX}). 
\par
A rigidity result for problem \eqref{eq:problem}-\eqref{extra-condition} has been proved by J.~Guo and C.~Xia in \cite{GX}. In our slightly different setting, the main result in \cite{GX} states that, if a suitably regular overdetermined solution exists, then $R>0$, $\Si$ must be the spherical cap 
defined by
\begin{equation}
\label{spherical-cap}
\{ x\in B_+: |x-z|=R\} \ \mbox{ with } \ |z|=\sqrt{1+R^2},
\end{equation}
and $u$ must be equal on $\ol{\Om}$ to the quadratic polynomial defined for $x\in\RR^N$ by 
$$
\frac12\,(|x-z|^2-R^2).
$$
We shall compute in Proposition \ref{prop:value-R} the exact value of $R$ in terms of $\Om$ as
\begin{equation}
\label{value-R-intro}
R=N\,\frac{\int_\Om x_N\,dx}{\int_\Si x_N\,dS_x}.
\end{equation} 

\begin{figure}[h]
\label{fig:symmetric-domain}
\begin{center}
\begin{tikzpicture}[scale=.32] 
\draw[thick] (7,0) arc (0.8:180:8); 
\draw[thick] (7,0) arc (-64.5:-114.5:19); 
\draw[thick,dotted] (7,0) arc (66:116:19); 

\draw plot [very thick, mark=*, smooth] coordinates
{(-4.6,15)}; 
\node at (-5.5,15.3) {$z$}; 

\draw plot [very thick, mark=*, smooth] coordinates
{(-1,0)}; 
\node at (-1,-.8) {$0$}; 

\draw plot [very thick, mark=*, smooth] coordinates
{(-1.95,4)}; 

\draw[thick] (-4.6,15) -- (-8.65,2.4); 
\draw[solid] (-1,0) -- (-8.65,2.4); 

\draw[thick] (-4.6,15) -- (4.7,5.45); 
\draw[solid] (-1,0) -- (4.7,5.47); 
\draw[thick] (-4.6,15) -- (-1,0); 

\draw [thick,domain=252:314.5] plot ({-4.6+13.2*cos(\x)}, {15+13.2*sin(\x)}); 
\draw [thick, domain=268.4:297] plot ({-7.9+27.55*cos(\x)}, {30.1+27.55*sin(\x)}); 

\draw [thick,dotted,domain=79.8:125] plot ({1.6+17.5*cos(\x)}, {-11.6+17.5*sin(\x)}); 

\node at (-10.5,8.5) {$\sqrt{1+R^2}$};  
\draw[->] (-8,8.5) -- (-3.4,9); 

\node at (3.5,2) {$\mathbf\Si$}; 
\draw[->] (3,2) -- (0.2,2.5); freccia Sigma

\node at (0.8,11.5) {$R$}; 
\node at (-4.5,0.5) {$1$}; 

\node at (-0.8,8.6) {$\mathbf T$}; 
\node at (0,5) {$\Om$}; 
\node at (7,4.5) {$B_+$}; 
\end{tikzpicture}
\end{center}
\caption{The construction of a symmetric domain $\Om$. The $R$-spherical cap $\Si$ meets orthogonally the unit spherical cap $T$.}
\end{figure}
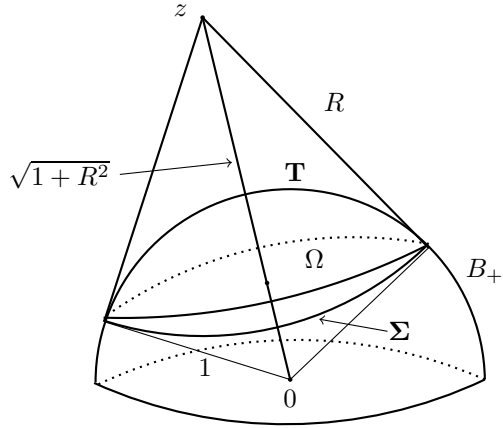

Thus, $\Si$ must be a spherical cap and $\Om$ results as a lenticular domain, as shown in Figure \ref{fig:symmetric-domain}.  In this paper, we shall study 
the problem \eqref{eq:problem}-\eqref{extra-condition} from a quantitative point of view. In other words, we will estimate how close $\Si$ is to a spherical cap in terms of the deviation of $u_\nu$ from the constant $R$ in some Lebesgue norm on $\Si$.
\par
In order to do that, we first refine Guo and Xia's rigidity result. In fact, by shadowing the arguments in \cite{GX}, we obtain the following integral identity for the solution of \eqref{eq:problem}:
\begin{equation}
\label{fundamental-identity}
\int_\Om x_N (-u) \left\{ | \na^2 u|^2- \frac{(\De u)^2}{N} \right\} dx = \\
\frac{1}{2} \int_\Si \left( u_\nu^2 - R^2 \right) \bigl[ u_\nu x_N - \lan X^q, \nu\ran \bigr]  dS_x.
\end{equation}
Here, $X^q$ is the conformal \textit{Killing field} defined by
\begin{equation}
\label{Killing}
X^q=x_N\, x -\frac{1}{2} \left( |x|^2 +1 \right) e_N=x_N\, \na q(x) -q(x)\, e_N, \ \ x\in\RR^N,
\end{equation}
where $e_N=(0,\dots, 1)\in\RR^N$, and we set: $q(x)=(1+|x|^2)/2$. 
Integral identities of this kind have been obtained for the Alexandrov's Soap Bubble Theorem and the classical Serrin's problem by the authors of this note (see \cite{MP, MP2, MP3, Pog2}). In those cases, the role of the field $X^q$ in the identity was played by the identity field $\RR^N\ni x\mapsto x$. Note that, on the unit sphere $S$, $X^q$ is the projection of $-e_N$ on the tangent space to $S$.
\par
In \cite{GX}, it is proved that, if $u$ satisfies \eqref{eq:problem}-\eqref{extra-condition}, then the left-hand side of \eqref{fundamental-identity} must be zero. Since $x_N>0$ in $\Om\subset B_+$ and $u<0$ in $\Om$ by \cite[Proposition 2.3]{GX}, the function in the braces at the left-hand side of \eqref{fundamental-identity}
must vanish identically on $\Om$, since it is always non-negative by the Cauchy-Schwarz inequality. As a by-product, one infers that $u$ must be a spherically symmetric quadratic polynomial, as noted in \cite{MP}. Thus, $\Si$ must be a portion of a sphere, since $u=0$ on $\Si$. The lenticular shape of $\Om$ then ensues quite easily.
\par
Now, observe that, from \eqref{fundamental-identity} it is evident that its right-hand side (and hence its left-hand side) is null if \eqref{extra-condition} holds. However, \eqref{fundamental-identity} gives more information for at least two reasons. One is that Guo and Xia's rigidity result can be merely obtained
under the weaker assumption that the right-hand side of \eqref{extra-condition} is non-positive.
The second and more important reason is that the identity gives  quantitative information. In fact, if we know that $u_\nu$ deviates from $R$ by little in some integral norm, then the integral at the left-hand side of \eqref{fundamental-identity} is small.  
\par
Now, notice that, if we consider a quadratic polynomial as defined by
$$
Q(x)=\frac12\,|x-z|^2-q_0 \ \mbox{ for } \ x, z\in\RR^N, \ q_0\in\RR,
$$
and we set $h=Q-u$, then it turns out that
$$
| \na^2 u|^2- \frac{(\De u)^2}{N}=|\na^2 h|^2.
$$

Thus, the square root of the first integral in \eqref{fundamental-identity} can be seen as the weighted (second order) $W^{2,2}$-seminorm in $\Om$ of $h$ with respect to the positive measure $x_N\,[-u(x)]\,dx$. Also, notice that $h=Q$ on $\Si$, and hence $Q$ has to do with the distance of the point $z$ to points in $\Si$. Therefore, we will see that, in order to obtain an estimate of closeness of $\Si$ from the spherical cap defined in \eqref{spherical-cap}, 
it is just the matter of proving that the oscillation of $h=Q$ on $\Si$ can be controlled in terms of the aforementioned weighted $W^{2,2}$-seminorm of $h$.
\par
We are now going to present our quantitative rigidity estimates. We need to recall some notation from the subsequent sections. 
\par
As in \cite{GX}, we assume that $u \in W^{1,\infty} (\Om) \cap W^{2,2} (\Om)$.
Under this assumption, since $u \in W^{1,\infty} (\Om)$ and $u=0$ on $\Si$, then $u \in C^{0,1} (\ol{\Om})$. In fact, we can extend $u$ by $0$ outside $\Om$ to the whole $B$,
thus obtaining a function in $W^{1, \infty} (B)$, which coincides with $C^{0,1} (\ol{B})\supset C^{0,1} (\ol{\Om})$, since $B$ is convex.
%
%
Thus, we let $L$ to be an upper bound\footnote{When  $\ol{\Si}$ and $\ol{T}$ intersect orthogonally, \cite[Proposition 3.5]{GX} ensures that $u \in C^{1,\ga} (\ol{\Om}) \cap W^{2,2} (\Om) $: their argument is based on spherical reflection.
The global $C^{1,\ga} (\ol{\Om})$ regularity of $u$ is also guaranteed whenever $\Si$ is a capillary surface with contact angle $\te \in ( 0, \pi/2 )$: see \cite[Theorem 3.2]{JXZ}.} of the Lipschitz seminorm defined in \eqref{eq:seminorm}, i.e. $L\ge [u]_{C^{0,1}(\ol{\Om})}$.

Also, 
we present our stability results under the assumption that $\ol{\Si}$ and $\ol{T}$ intersect on $\La$ in a way that $\Om$ satisfies the $(\te,a)$-uniform interior cone condition, for given parameters $\te$ and $a$ (see Section \ref{sec:preliminary-estimates} for the definition).
We adopt this condition to avoid an excessively technical presentation.
Nevertheless, our arguments could be adapted and the same stability result of Theorem \ref{thm:Improvement1} below achieved in
more general cases (see Remark \ref{rem:John extension}).
\par
In order to measure the deviation of $\Si$ from a spherical cap, for a given point $z \in \RR^N$, we define two quantities,
\begin{equation*}
	\rho_e= \max_{x \in \ol{\Si} }{|x-z|} \quad\mbox{and}\quad \rho_i=\min_{x \in \ol{\Si} }{|x-z|},
\end{equation*}
so that we have:
$$
\ol{\Si} \subseteq \left[ \ol{B}_{\rho_e}(z) \setminus B_{\rho_i}(z) \right] \cap \ol{B}_+.
$$
The point $z$ must be conveniently chosen. A good choice of $z$ is a somewhat modified center of mass of $\Om$:
\begin{equation}
\label{eq:choice z old school}
	z = \frac{1}{|\Om|} \left\{ \int_\Om x \, dx - \int_T u(x) \, x \, dS_x  \right\}.
\end{equation}
With this choice, we have that the mean value of the field $\na h$ is zero. 
This will allow the use of certain suitable Hardy-Poincar\'e-type inequalities.
\par

We now present our stability results

Our most general quantitative estimates are contained in Theorem \ref{thm:stability-general}.
Here, we prefer to present three special instances of that result in three relevant situations, which better depict the dependence of the estimates on certain geometrical assumptions on the surface $\Si$. 
\par
In the next theorem, $\Si$ is not allowed to touch the flat part of $B_+$.

\begin{thm}[$\Si$ does not touch $\pa B_+\setminus\pa B$]
\label{thm:Improvement1}
Set $N\ge 2$. Let $\Om$ be a domain contained in $B_+$ and satisfying
the $(\te,a)$-uniform interior cone condition. Assume that there exists a positive number $m$ such that 
\begin{equation}
\label{dont-touch}
\ol{\Om}\subset\{ x\in \ol{B}_+: x_N\ge m\}.
\end{equation}
\par
Let $u \in W^{1,\infty}(\Om)\cap W^{2,2} (\Om)$ be the solution of \eqref{eq:problem} and assume that $L\ge [u]_{C^{0,1}(\ol{\Om})}$.
Moreover, let $R$ and $z$ be the number and point defined in \eqref{value-R-intro} and \eqref{eq:choice z old school}.
Then, it holds that
	\begin{equation*}
		\rho_e-\rho_i\le c \,
		\begin{cases}
			\nr u_\nu^2 - R^2 \nr_{1,\Si}^{1/2} \, \max \left\{ \log \left(\nr u_\nu^2 - R^2 \nr_{1,\Si}^{- 1/2 }  \right) , 1 \right\} \quad & \text{ for } N =2 ,
			\\
			\nr u_\nu^2 - R^2 \nr_{1,\Si}^{1/N} \quad & \text{ for } N \ge 3 ,
		\end{cases}
	\end{equation*}
for some non-negative constant $c=c(N, \te, a ,  L , m )$.
\end{thm}

In Section \ref{subsec:general stability} we show that the assumption \ref{dont-touch} can be removed at the cost of getting a slightly worse stability exponent, namely $1/(N+1)$ in place of $1/N$ for $N \ge 3$ (see Theorem \ref{thm:stability-general}). Such a generalization is non-trivial and requires a new and careful analysis, which is provided in Section \ref{subsec:general stability}.

The next result considers the case where $\Om$ satisfies an interior sphere condition relative to $B_+$. In fact, the same stability rate of Theorem \ref{thm:Improvement1} can also be obtained if \eqref{dont-touch} is dropped and replaced by the assumption that $\Om$ satisfies the \textit{strong $r_i$-uniform interior sphere condition relative to $B_+$}. Such a condition, which is introduced in Section \ref{subsec:some geometrical facts} following the spirit of \cite[Section 4.1]{Pog3}, is surely satisfied whenever $\ol{\Si}$ and $\ol{T}$ intersect orthogonally. 

\begin{thm}[$\Om$ satisfies a strong sphere condition]
\label{thm:Improvement2}
Set $N\ge 2$ and let $\Om$ be a domain contained in $B_+$. Assume that $\Om$ satisfies the $(\te,a)$-uniform interior cone condition and the strong $r_i$-uniform interior sphere condition relative to $B_+$.
\par
Let $u\in W^{1,\infty}(\Om)\cap W^{2,2} (\Om)$ be the solution of \eqref{eq:problem} and assume that $L\ge [u]_{C^{0,1}(\ol{\Om})}$.	Moreover, let $R$ and $z$ be the number and point defined in \eqref{value-R-intro} and \eqref{eq:choice z old school}. Then, it holds that
	\begin{equation*}
		\rho_e-\rho_i\le c \,
		\begin{cases}
			\nr u_\nu^2 - R^2 \nr_{1,\Si}^{1/2} \, \max \left\{ \log \left(\nr u_\nu^2 - R^2 \nr_{1,\Si}^{-1/2 }  \right) , 1 \right\} \quad & \text{ for } N =2 ,
			\\
			\nr u_\nu^2 - R^2 \nr_{1,\Si}^{1/N}  \quad & \text{ for } N \ge 3 ,
		\end{cases}
	\end{equation*}
for some non-negative constant $c=c(N, \te, a, L , r_i)$.
\end{thm}

The rate of stability further improves if both additional assumptions are in force.

\begin{thm}[$\Om$ satisfies a strong sphere condition and $\Si$ does not touch $\pa B_+\setminus\pa B$]
\label{thm:Improvement3}
	Set $N\ge 2$ and let $\Om$ be a domain contained in $B_+$. Assume that $\Om$ satisfies the $(\te,a)$-uniform interior cone condition and the strong $r_i$-uniform interior sphere condition relative to $B_+$. In addition, suppose that there exists $m>0$ such that \eqref{dont-touch} holds.
\par
Let $u\in W^{1,\infty}(\Om)\cap W^{2,2} (\Om)$ be the solution of \eqref{eq:problem} and assume that $L\ge [u]_{C^{0,1}(\ol{\Om})}$. 
Moreover, let $R$ and $z$ be the number and point defined in \eqref{value-R-intro} and \eqref{eq:choice z old school}. Then, it holds that 
	\begin{equation*}
		\rho_e-\rho_i\le c \,
		\begin{cases}
			\nr u_\nu^2 - R^2 \nr_{1,\Si}^{ 1/2 } \quad & \text{ for } N=2 ,
			\\
			\nr u_\nu^2 - R^2 \nr_{1,\Si}^{ 1/2 } \, \max \left\{ \log \left( \nr u_\nu^2 - R^2 \nr_{1,\Si}^{ -1/2 }  \right) , 1 \right\}  \quad & \text{ for } N=3 ,
			\\
			\nr u_\nu^2 - R^2 \nr_{1,\Si}^{  1/(N-1) } \quad & \text{ for } N\ge 4 ,
		\end{cases}
	\end{equation*}
$c=c(N, \te, a , L , r_i , m)$.
\end{thm}
\par
The paper is organized as follows. In Section \ref{sec:fundamental}, we derive our fundamental integral identity \eqref{fundamental-identity}. In Section \ref{sec:preliminary-estimates}, we prepare the proofs of Theorems  \ref{thm:Improvement1}--\ref{thm:Improvement3} and \ref{thm:stability-general},
by collecting a pointwise estimate from below for $-u$ in terms of the distance of a point $x$ to the boundary $\Ga$ and 
some Poincar\'e-type estimates in weighted spaces. 
These adapt to the constrained case $\Om\subset B_+$ similar bounds obtained in \cite{MP, MP2, MP3} (see also \cite{DPV}).
Finally, in Section \ref{sec:stability}, we carry out the proofs of Theorems  \ref{thm:Improvement1}--\ref{thm:Improvement3} and \ref{thm:stability-general}.

\section{A fundamental identity}
\label{sec:fundamental}

In this section, we shall prove the identity \eqref{fundamental-identity}. 

For later use, we preliminarly recall some easily verified properties of the Killing field $X^q$ defined in \eqref{Killing} and the solution $u$ of \eqref{eq:problem}.
In fact, it holds that
\begin{equation}
\label{Killing-properties}
\dv X^q= N x_N \ \mbox{ in } \RR^N; \ \  X^q= x_N x - e_N, \
\lan X^q, \nu\ran=\lan X^q, x\ran=0 \ \text{ on } \ S,
\end{equation}
\begin{equation}
\label{u-further-properties}
\na (\De u)=0 \ \mbox{ in } \ \Om, \quad \na u=u_\nu\,\nu \ \mbox{ on } \ \Si, \quad
\lan \na^2 u\,\nu, \om\ran=0 \ \mbox{ on } \ T,
\end{equation}
for every direction $\om$ which is tangential to $T$. The last two conditions follow from the fact that $\Si$ and $T$ are level surfaces for $u$ and $u_\nu-u=\lan x,\na u\ran-u$.
\par
The proof of \eqref{fundamental-identity} is inspired by calculations carried out in \cite{GX}. 
Essentially, those are a combination of repeated integrations by parts and the application of conditions \eqref{Killing-properties} and \eqref{u-further-properties}.

We begin by adapting to our aims and notations an identity in \cite[Proposition 3.3]{GX}. We introduce the so-called \textit{$P$-function} by setting:
\begin{equation}
\label{P-function}
P= \frac{1}{2}\, |\na u|^2 -u \ \mbox{ in } \ \Om.
\end{equation}

\begin{lem}[A Pohozaev-type identity]\label{lem:Pohozaev type}
Let $u\in W^{1,\infty}(\Om) \cap W^{2,2} (\Om)$ be the solution of \eqref{eq:problem}. Then, the following identity holds:
\begin{equation}
\label{eq:Pohozaevadhoc}
N \int_\Om x_N P\, dx = \frac{1}{2} \int_\Si u_{\nu}^2 \lan X^q, \nu\ran \, dS_x.
\end{equation}
\end{lem}
\begin{rem}
	{\rm
		Being as $u \in W^{1,\infty}(\Om) \cap W^{2,2} (\Om)$, all the integration by parts performed in this section are allowed (see, e.g., the version of the divergence theorem stated in \cite[Proposition 3.2]{GX}).
	}
\end{rem}
\begin{proof}[Proof of Lemma \ref{lem:Pohozaev type}]
The proof of \cite[Proposition 3.3]{GX} can be summarized and reorganized as follows.
By straightforward computations, we see that the following differential identity holds true:
\begin{multline*}
N\,x_N\,P=\dv\left\{\lan X^q, \na u\ran \,\na u-N\,u\,X^q-\frac12\,|\na u|^2 X^q\right\}+ \\
(N-1)\,\dv\left\{x_N\,u\,\na u-\frac12\, u^2 e_N\right\}.
\end{multline*}
\par
Next, we integrate on $\Om$ and use the divergence theorem. We have that
\begin{multline*}
N\,\int_\Om x_N\,P\,dx=\int_\Si\lan X^q,\na u \ran\,u_\nu\,dS_x+\int_T \lan X^q, \na u\ran\,u_\nu\,dS_x+ \\
-\frac12\int_\Si u_\nu^2 \lan X^q, \nu \ran\,dS_x+ 
(N-1) \int_T x_N\,u\,u_\nu\,dS_x-\frac12\,(N-1) \int_T x_N\,u^2 dS_x.
\end{multline*}
Here, we have used that $u=0$ on $\Si$ and $\lan X^q, \nu\ran=0$ on $T$. 
\par
Now, we use that $\na u=u_\nu\,\nu$ on $\Si$ and $u_\nu=u$ on $T$, and hence infer that
\begin{multline*}
N\,\int_\Om x_N\,P\,dx= \\
\frac12 \int_\Si u_\nu^2 \lan X^q,\nu \ran\,dS_x +
\int_T \lan - (e_N)_T, \na_T u \ran\,u\,dS_x+ 
\frac12\,(N-1) \int_T x_N\,u^2 dS_x .
\end{multline*}

Here, we have also noticed that
\begin{multline*}
\lan X^q, \na u\ran\,u_\nu=\lan X^q, u_\nu\,\nu+\na u-u_\nu\,\nu\ran\,u= \\
\lan X^q, \na u-u_\nu\,\nu\ran\,u=  \lan - (e_N)_T, \na_T u\ran\,u \ \mbox{ on } \ T .
\end{multline*}
where with $ (e_N)_T$ and $\na_T u$ we denote the tangential components of $e_N$ and $\na u$ on $T$.

\par
Thus, we are left to prove that the two integrals on $T$ sum up to zero. 
This ensues by applying the divergence theorem on the surface $T$:
\begin{multline*}
0=\int_\La u^2 \lan  (e_N)_T , \nu_\La\ran\,d\ell_x=\int_T {\dv}_T (u^2  (e_N)_T  )\,dS_x= \\
\int_T \left\{u^2 {\dv}_T \left( (e_N)_T  \right) + 2\,\lan (e_N)_T , \na_T u\ran\,u\right\}dS_x.
\end{multline*}
Here, ${\dv}_T$ denotes the tangential divergence.
The first integral is zero, because $u=0$ on $\La$.
The conclusion follows by noting that ${\dv}_T \left( (e_N)_T \right)= - (N-1)\,x_N$. 
\end{proof}

We are now ready to prove the main result of this section.

\begin{thm}[Fundamental identity]
\label{th:fundamental-identity-c}
Let $u\in W^{1,\infty}(\Om) \cap W^{2,2} (\Om)$ be the solution of \eqref{eq:problem}. Then, for any given constant $c$, the following identity  holds:
\begin{equation}
\label{fundamental-identity-c}
\int_\Om x_N (-u) \left\{ | \na^2 u|^2- \frac{(\De u)^2}{N} \right\} dx = \
\frac{1}{2} \int_\Si \left( u_\nu^2 - c^2 \right) \bigl[ x_N\,u_\nu - \lan X^q, \nu\ran \bigr]  dS_x.
\end{equation}
\end{thm}
\begin{proof}
Taking the vector field $X^u=x_N\,\na u-u\,e_N$, we compute that
\begin{equation}
	\label{Xu-properties}
	\dv(X^u)=N\,x_N \ \mbox{ in } \ \Om, \quad \lan X^u, \nu\ran=0 \ \mbox{ on } \ T, \quad
	\lan X^u, \nu\ran=x_N\,u_\nu \ \mbox{ on } \ \Si ,
\end{equation}	
and hence, by the divergence theorem and \eqref{Killing-properties}, 
$$
0=\int_\Om\dv( X^u - X^q )\,dx=\int_\Ga\lan X^u -  X^q , \nu\ran\,dS_x=
\int_\Si [x_N\,u_\nu-\lan X^q,\nu\ran]\,dS_x .
$$
Thus, it is sufficient to prove \eqref{fundamental-identity-c} for $c=0$.
\par
Next, observe that
\begin{equation*}
\De P = | \na^2 u|^2 -  \frac{(\De u)^2}{N},
\end{equation*}
and hence, the Gauss-Green formula gives:
\begin{multline*}
\int_\Om x_N u\, \De P \, dx =\int_\Om \De (x_N u)\,P\,dx+\int_\Om \dv\bigl\{ x_N u\,\na P-P\,\na (x_N u)\bigr\}\,dx=
\\
\int_\Om [2\,u_{x_N}+x_N\,\De u]\,P\,dx + \int_\Ga \bigl\{ x_N u\,P_\nu-[\lan e_N,\nu\ran\,u+x_N\,u_\nu]\,P\bigr\} dS_x= \\
\int_\Om [2\,u_{x_N}+N\,x_N]\,P\,dx -\int_\Si x_N u_\nu P \, dS_x
+\int_T \bigl\{ x_N u \, P_\nu-x_N\,(u+u_\nu) P  \bigr\}  dS_x.
\end{multline*}
Here, we used that $\De u=N$ in $\Om$, $u=0$ on $\Si$, and $\lan e_N,\nu(x)\ran=x_N$ for $x\in T$. 
Consequently, we deduce that
\begin{multline*}
\int_\Om x_N u\, \De P \, dx =
2 \int_\Om u_{x_N}  P \, dx + \frac{1}{2} \int_\Si \lan X^q, \nu\ran u_{\nu}^2 \, dS_x+\\
 -\frac12 \int_\Si x_N u^3_\nu\, dS_x
+\int_T x_N u \, P_\nu\,dS_x-2\,\int_T x_N u\, P \, dS_x,
\end{multline*}
since $u_\nu=u$ on $T$ and $u=0$ and $|\na u|=u_\nu$ on $\Si$. Here, the second summand at the right-hand side is obtained by applying \eqref{eq:Pohozaevadhoc}.
All in all, we have that
\begin{multline*}
\int_\Om x_N u\, \De P \, dx =
-\frac12 \int_\Si u^2_\nu\, \bigl[x_N u_\nu-\lan X^q, \nu\ran\bigr]\, dS_x+ \\
2 \int_\Om u_{x_N}  P \, dx+\int_T x_N u \, P_\nu\,dS_x-2\,\int_T x_N u\, P \, dS_x,
\end{multline*}
and hence we are left to prove that the last three integrals sum up to zero.
\par
The integral on $\Om$ can be treated by integrating on $\Om$ the differential identity:
$$
\dv\left\{ \left[ (2\,u\,P+u^2)\,I-u^2 \na^2u \right] e_N \right\}  =2\,u_{x_N} P.
$$
Here, $I$ denotes the $N\times N$ identity matrix.
In this calculation, we have used the first identity in \eqref{u-further-properties}. Thus, by the definition of $P$ and divergence theorem, we get:
\begin{equation}
\label{volume-integral}
2 \int_\Om u_{x_N}  P \, dx=
\int_T \bigl\{x_N\,\bigl[u\,|\na u|^2-u^2\bigr] -u^2\,\lan\na^2 u\, e_N,\nu\ran\bigr\} dS_x.
\end{equation}
Again, we used that $u=0$ on $\Si$ and $\lan e_N,\nu(x)\ran=x_N$ for $x\in T$.
%

%

\par
Next, we directly compute on $T$ that
$$
P_\nu = \lan \na^2 u \, \na u, \nu\ran - u_\nu =\lan \na^2 u \, (u_\nu \nu+\om), \nu\ran - u_\nu= 
 u \,\lan\na^2 u \, \nu, \nu\ran - u.
$$
In the first equality, we have decomposed $\na u$ into the sum of its normal and tangential components $u_\nu\,\nu$ and $\om$. In the second equality, we used the third identity in \eqref{u-further-properties} and 
that $u_\nu = u$ on $T$.
Moreover, we observe that on $T$ it holds that
$$
\lan \na^2 u\,e_N, \nu\ran=\lan \na^2 u\, \left(  x_N \, \nu - X^q \right) , \nu\ran=x_N\,\lan\na^2 u \, \nu, \nu\ran,
$$
by the second identity in \eqref{Killing-properties} (being as $\nu(x)=x$ on $S$) and the third identity in \eqref{u-further-properties} (being as $X^q$ tangent to $T$).
\par
Therefore, with this and the identity for $P_\nu$ in mind, we finally conclude that
$$
2 \int_\Om u_{x_N}  P \, dx+\int_T x_N u \, P_\nu\,dS_x-2\,\int_T x_N u\, P \, dS_x=0,
$$
thanks to
\eqref{volume-integral}. This was what we were left to prove.
\end{proof}

A convenient choice of the constat $c$ in \eqref{fundamental-identity-c} is suggested by the following proposition. 

\begin{prop}[The value of $R$]
\label{prop:value-R}
Let $u \in W^{1,\infty}(\Om) \cap W^{2,2} (\Om)$ be the solution of~\eqref{eq:problem}.  
\par
If $u_\nu=R$ on $\Si$, then we have that
\begin{equation}
\label{value-R}
R=N\,\frac{\int_\Om x_N\,dx}{\int_\Si x_N\,dS_x}
=\frac{N\,|\Om|}{|\Si|}\,\frac{c_N^\Om}{c_N^\Si},
\end{equation}
where $c_N^E$ denotes the $N$-th coordinate of the center of mass of a set $E$. 
\end{prop}

\begin{proof}
By using the divergence theorem and \eqref{Xu-properties}, we compute that
$$
N\,\int_\Om x_N\,dx=\int_\Om\dv(X^u)\,dx=\int_\Ga \lan X^u, \nu\ran\,dS_x=
\int_\Si x_N\,u_\nu\,dS_x=R\,\int_\Si x_N\,dS_x .
$$
Thus, \eqref{value-R} follows at once.
\end{proof}

As a consequence of this proposition and Theorem \ref{th:fundamental-identity-c}, we obtain a more general version of Guo and Xia's rigidity result. 

\begin{cor}
Let $u \in W^{1,\infty}(\Om) \cap W^{2,2} (\Om)$ be the
solution of \eqref{eq:problem}.  
\par 
If the right-hand side of \eqref{fundamental-identity-c} is non-positive for some $c\in\RR$, then 
$$
u(x)=\frac12\,\left(|x-z|^2-R^2\right) \ \mbox{ for } \ x\in\Om
$$
and $\Si$ must be the spherical cap 
$\{ x\in B_+: |x-z|=R\}$,
where $R$ is given by \eqref{value-R}, and $z=(z', z_N)$ is such that
$|z|=\sqrt{1+R^2}$ and $|z'|\le 1$.
The same conclusion holds true, in particular, if $u_\nu$ is constant on $\Si$. 
\end{cor}

\begin{proof}
By Theorem \ref{th:fundamental-identity-c}, our assumption clearly gives that the volume integral at the left-hand side of \eqref{fundamental-identity-c} must be zero. Since $u<0$ in $\Om$ by \cite[Proposition 2.3]{GX} and $x_N>0$ in $\Om\subset B_+$, we infer that 
$$
0\equiv |\na^2 u|^2- \frac{(\De u)^2}{N}=|\na^2 u|^2- \frac{\lan\na^2 u, I\ran^2}{N} \ \mbox{ in } \ \Om.
$$
Thus, the Cauchy-Schwarz inequality for the $N^2$-vectors $\na^2 u$ and $I$ holds with the sign of equality. As already observed in \cite{MP2}, we have that $u$ must be a quadratic polynomial of the form:
$$
u(x)=\frac12\, \left( |x-z|^2-q_0 \right) \ \mbox{ for some } \ q_0\in\RR.
$$
Since $u=0$ on $\Si$, we infer that $ q_0 >0$ and $\Si$ must equal $\{x\in B_+: |x-z|=\sqrt{q_0}\}$ --- a spherical cap.
\par
We now determine $q_0$ and $z$.
On one hand, observe that
\begin{multline*}
N\int_\Om x_N\,dx=\int_\Om \dv (X^u)\,dx=\int_\Ga \lan X^u, \nu\ran\,dS_x= \\
\int_\Si x_N\,u_\nu\,dS_x=\int_\Si x_N\,|x-z|\,dS_x=\sqrt{q_0}\,\int_\Si x_N\,dS_x,
\end{multline*}
i.e. we have that $q_0=R^2$. In particular, we infer that $u_\nu=R$ on $\Si$.
On the other hand, for $x\in T$, we must have that
$$
0=u_\nu(x)-u(x)=\lan x-z, \nu\ran-\frac12\,(|x-z|^2- q_0 )=\frac12\,(1+ q_0 -|z|^2),
$$
being as $\nu(x)=x$ for $x\in T$. Hence, $|z|=\sqrt{1+q_0}=\sqrt{1+R^2}$. Finally, we have that $|z'|\le 1$, since $\ol{T}$ is required to be contained in the upper hemisphere of $\pa B_+$.
\par
If $u_\nu$ is constant on $\Si$, then Proposition \ref{prop:value-R} tells us that the constant must equal the number $R$ in \eqref{value-R}. Choosing $c=R$ gives the the right-hand side of \eqref{fundamental-identity-c} is zero.
\end{proof}

\begin{rem}
{\rm
It is just an exercise to check that any spherical cap of the form specified in the corollary meets $T$ orthogonally.
}\end{rem}

\section{Weighted Sobolev-type bounds }
\label{sec:preliminary-estimates}

In this section, we collect some notations, definitions, and preliminary lemmas. We will provide the proofs only when they are not available in the literature.
\par
Given $\theta \in \left( 0, \pi /2 \right]$ and $a >0$, we say that a set $E$ satisfies the \textit{$(\theta, a )$-uniform interior cone condition}, if for every $x \in \pa E$ there is a unit vector $\om=\om_x$ such that the cone with vertex at the origin, axis $\om$, opening width $\te$, and height $a$ defined by
$$
\cC_{\om}=\left\{ y \, : \, \langle y , \om \rangle > |y| \cos(\theta) , \, |y|<a \right\}
$$
is such that
\begin{equation}\label{def:cone strong}
w + \cC_{\om}  \subset E \ \text{ for every } \ w \in B_a (x) \cap \ol{E} .
\end{equation}
Such a condition is equivalent to Lipschitz-regularity of the domain; more precisely, it is equivalent to the strong local Lipschitz property of Adams \cite[Pag 66]{Ad} and to the uniform Lipschitz regularity in \cite[Section III]{Ch} and \cite[Definition 2.1]{Ru}.
\par
In the sequel, we shall always consider a domain $\Om\subset B_+$ that satisfies this cone condition.
We then denote by $C^{0, 1} (\ol{\Om})$ the class of Lipschitz continuous functions on $\ol{\Om}.$  If $u\in C^{0, 1} (\ol{\Om})$, we set $L$ to be the Lipschitz constant of $u$ in $\ol{\Om}$, i.e.
\begin{equation}
\label{eq:seminorm}
	L = \left[ u \right]_{C^{0,1}(\ol{\Om})} = \sup \left\{ \frac{|u(x_1) - u(x_2)|}{|x_1 - x_2|} \, : \,  x_1,x_2 \in \ol{\Om}, \, x_1 \neq x_2 \right\}.
\end{equation}


The Hardy-Poincar\'{e}-type inequalities in the lemma and corollary below are adapted from \cite[Section 3.2]{Pog} and \cite[Lemma 2.1]{MP3} and can be deduced by the works of Bojarski \cite{Bo} and Hurri-Syrj\"anen \cite{HS1, HS}.  For a domain $E\subset \RR^N$, we denote by $d_E$ its diameter.


\begin{lem}
\label{lem:John-two-inequalities}
	Let $E \subset \RR^N$ be a bounded domain satisfying the $(\te,a)$-uniform interior cone condition.
\par
Consider three numbers $r, p, \al$ such that, either
	\begin{equation}\label{eq:conditionHS}
		1 \le p \le r \le \frac{N\,p}{N-p\,(1 - \al )} , \quad p\,(1 - \al)<N , \quad 0 \le \al \le 1 ,
	\end{equation}
	or
	\begin{equation}\label{eq:conditionBS}
		r = p \in \left[ 1, \infty \right) , \quad \al=0.
	\end{equation}
	Then, there exists a positive constant, $c=c(N, r, p, \al, \te, a, d_E)$ such that
	\begin{equation}
		\label{John-harmonic-poincare}
		\nr f - f_E \nr_{r,E} \le c\, \nr \de_{\pa E}^{\al} \, \na f  \nr_{p, E},
	\end{equation}
	for every function $f \in L^1_{loc}(E)$ such that $\de_{\pa E}^{\al} \, \na f \in L^p (E)$. Here, $f_E$ denotes the mean value of $f$ on $E$.
	
	If $E \subset B_+$, the dependence of $c$ on $d_E$ can be removed, being as $d_E \le 2$.
\end{lem}
\begin{cor}\label{cor:JohnPoincareaigradienti}
	Let $E\subset\RR^N$, $N\ge 2$, be a bounded domain satisfying the $(\te,a)$-uniform interior cone condition and let $f$ be a function such that
	 $\na f \in L^1_{loc}(E)$ and $\de_{\pa E}^{\al} \, \na^2 f \in L^p(E)$. Consider three numbers $r, p, \al$ 
	satisfying either \eqref{eq:conditionHS} or \eqref{eq:conditionBS}.
	If
	$$\int_E \na f \, dx = 0,$$
	%
	%
	then it holds that
	\begin{equation*}
		\nr \na f \nr_{r, E} \le c \, \nr \de_{\pa E}^{\al} \, \na^2 f  \nr_{p, E},
	\end{equation*}
	where $c$ is the same constant appearing in \eqref{John-harmonic-poincare}.
\end{cor}

\begin{rem}[On the proof of Lemma \ref{lem:John-two-inequalities} and Corollary \ref{cor:JohnPoincareaigradienti}]\label{rem:on the proof of John and Poincare}
{\rm	
Lemma \ref{lem:John-two-inequalities} and Corollary \ref{cor:JohnPoincareaigradienti} hold true in the more general case where $E$ is a John domain: we refer the reader \cite[proof of item(i) of Lemma 2.1 and item (i) of Corollary 2.3]{MP3}) for details. 
	Roughly speaking, a domain is a $b$-John domain if it is possible to travel from one point of the domain to another without going too close to the boundary (see Section \ref{appendix:cone and John} for the precise definition).
The class of John domains contains Lipschitz domains but also very irregular domains with fractal boundaries as, e.g., the Koch snowflake.
	
For $b$-John domains, (see \cite[items (i),(ii) of Remark 2.4]{MP3}), the following explicit bounds for the constant $c$ hold true:
	\begin{eqnarray*}
&&c \le k_{N,\, r, \, p,\, \al} \, b^N |E|^{\frac{1-\al}{N} +\frac{1}{r} - \frac{1}{p} } , \quad \text{if $r, p, \al$ are as in \eqref{eq:conditionHS}}, \\
&&c \le k_{N, \, p} \, b^{3N(1 + \frac{N}{p})} \, d_E , \quad \text{if $r,p, \al$ are as in \eqref{eq:conditionBS}}.
	\end{eqnarray*}
Of course, the volume appearing in the first inequality can be easily estimated by means of $|E| \le |B|\, d_E^N$.
Moreover, as we show in Lemma \ref{lem:cone condition strong implies John}, if a domain $E$ satisfies the $(\te,a)$-uniform interior cone condition, then it is a $b$-John domain and $b$ can be explicitly estimated in terms of $a, \te, d_E$ only.

We thus obtain that \eqref{John-harmonic-poincare} holds true with some constant $c$ that
%
%
depends only on $N, r, p, \al, \te, a, d_E$. 
If $E \subset B_+$, the dependence on $d_E$ can be removed, being as $d_E \le 2$.
}
\end{rem}

We conclude this section by providing an adaptation of \cite[Theorems 2.4 and 2.7]{MP6} (see also the errata corrige in Section \ref{sec:errata corrige}). We warn the reader that in \cite{MP6} we adopted a different normalization in the definition of the $L^p$-type norms.


\begin{lem}
	\label{lem:p>N + p<N and p=N} 
	Let $1 \le p < q\le \infty$.
	Let $E \subset \RR^N$ be a bounded domain satisfying the $(\te, a)$-uniform interior cone condition. 
	
	\begin{enumerate}[(i)]
	\item If $p>N$, then there is a non-negative constant $c=c (N, p, \te, a , d_E)$ such that
	\begin{equation*}
	\max\limits_{\ol{ E}} f - \min\limits_{ \ol{ E}} f \le c \, \nr \na f \nr_{p, E} ,
	\end{equation*}
	for any $f\in W^{1,q}(E)$.
	
	\item If $1\le p\le N $ and
$$
\al_{p,q}=\frac{p\, (q-N)}{N\,(q-p)},
$$
then there is a non-negative constant $ c = c (N, p, q, \te, a, d_E)$ such that
	\begin{equation*}
		\max\limits_{\ol{ E}} f - \min\limits_{ \ol{ E}} f \le 
		c \,
		\begin{cases}
		\nr\na f\nr_{p, E}^{\al_{p,q}}\nr\na f\nr_{q,E}^{1-\al_{p,q}} \quad & \text{ if } 1 \le p<N,
		\\
		\nr \na f \nr_{N, E } \log\left( e \, |E|^{\frac{1}{N} - \frac{1}{q}} \frac{  \nr \na f \nr_{q, E} }{\nr \na f \nr_{N, E} }\right) \quad & \text{ if } p=N ,
		\end{cases}
	\end{equation*}
for any $f\in W^{1,q}(E)$.	
\end{enumerate}
\medskip
Explicit bounds for the constants $c$ can be computed.  

If $E \subset B_+$, the dependence of the constants $c$ on $d_E$ can be removed, being as $d_E \le 2$.
\end{lem}

\begin{rem}\label{rem:0 sub-harmonic version}
{\rm
For sub-harmonic functions, a similar estimate in the case where $1 \le p <N$ and $q=\infty$ can also be obtained by putting together \cite[Lemma 3.14]{Pog2} with the Hardy-Poincar\`e-type inequalities mentioned in Lemma \ref{lem:John-two-inequalities}. See also \cite[Theorems 3.1 and 3.2]{MP4} for adaptations to either domains satisfying a weaker cone-type condition or John-type domains. 
}
\end{rem}

\section{Quantitative stability results}
\label{sec:stability}

In this section, we shall give the proofs of Theorems \ref{thm:Improvement1}--\ref{thm:Improvement3} and of the more general Theorem \ref{thm:stability-general} below. We begin by recalling some notations and other facts from the Introduction.

\subsection{Some geometrical facts}\label{subsec:some geometrical facts}

Lemma \ref{lem:relationdist} below is an adaptation of \cite[Lemma~3.1]{MP2} to the case of the mixed problem \eqref{eq:problem}, which takes inspiration from \cite[Section~4.1]{Pog3}.
We also mention that a fractional version of \cite[Lemma 3.1]{MP2} has been recently used in \cite{CDPPV}.


We will make use of the following maximum principle for mixed boundary value problems in $B_+$, which is a reformulation of \cite[Proposition 2.3]{GX}.

\begin{lem}\label{lem:comparison in ball}
	Let $f \in C^2(\Om) \cap C^1(\ol{\Om}\setminus \La )$ satisfy
	\begin{equation*}
		\De f \ge 0 \  \text{ in } \ \Om, \quad
		f \le 0    \ \text{ on } \ \ol{\Si}, \quad
		f_\nu \le f \ \text{ on } \ T,
	\end{equation*}
	and assume that $f \in W^{1,\infty} (\Om) \cap W^{2,2}(\Om)$. Then, we have that $f \le 0 $ on $\ol{\Om}$.
\end{lem}

\begin{proof}
	Set $f_+$ to be the positive part of $f$. By using the boundary conditions on $f$, an integration by parts, and the inequality $\De f \ge 0$ in $\Om$, we find that
	\begin{equation*}
		\int_{\Om} |\na f_+|^2 \, dx - \int_{T} f_+^2 \, dS_x \le
		\int_{\Om} |\na f_+|^2 \, dx - \int_{T \cup \Si} f_+ f_\nu \, dS_x = - \int_{\Om} f_+ \De f \, dx \le 0 .
	\end{equation*}
	On the other hand, \cite[(2.5)]{GX} informs us that
	$$
	0\ge\int_{\Om} |\na f_+|^2 \, dx - \int_{T} f_+^2 \, dS_x \ge \la_1 \int_{\Om} f_+^2 \, dx \ge 0,
	$$
	where $\la_1$ is the first Robin-Dirichlet eigenvalue.
	Hence, we easily infer that $f_+ \equiv 0$ in $\ol{\Om}$.
\end{proof}

In the spirit of \cite[Section 4.1]{Pog3}, we now introduce some appropriate sphere conditions peculiar to $B_+$.
In fact, we say that $\Om$ satisfies the \textit{$r_i$-uniform interior sphere condition relative to $B_+$}, if for each $x \in \ol{\Si}$ there exists a touching ball $B_{r_i}(x_0)$
of radius $r_i$ such that: (i) its center $x_0$ satisfies $|x_0|^2 \le 1 + r_i^2$
and (ii) its closure intersects
%
%
$\ol{B}_+ \setminus \Om$ only at $x$.

Notice that the requirements of this definition are related to how $\ol{\Si}$ and $\ol{T}$ intersect. In fact, a necessary condition for the validity of (i) and (ii) is that $\langle \nu_{\Si} (x) , \nu_{\pa B}(x) \rangle \ge 0$ for $x \in \pa \Si$.

Since in our setting $\Si$ is smooth, we must have that 
$x_0=x- r_i \, \nu(x)$ for $x \in \Si$.  
However, notice that this may not be the only possibility for the points on $\pa \Si$.

We say that $\Om$ satisfies the {\it strong $r_i$-uniform interior sphere condition relative to $B_+$} if, besides satisfying the $r_i$-uniform interior sphere condition relative to $B_+$, $\Om$ has the property that,
for any $x \in \ol{\Om}$ such that its closest point $\ul{x}$ to $\ol{\Si}$ belongs to $\pa \Si$,
the ball with radius $r_i$ and centered at the point $\ul{x} + r_i \frac{x-\ul{x}}{|x- \ul{x}|}$
is a touching ball at $\ul{x}$ relative to $B_+$ (as in the previous definition).
We notice that this condition is surely satisfied if $\ol{\Si}$ and $\ol{T}$ intersect orthogonally. 
\par
Here and in the sequel, $\de_A(x)$ will denote the distance of a point $x\in\RR^N$ to a set $A$.

\begin{lem}[A geometric bound]
	\label{lem:relationdist}
	Let $u$ be the solution of \eqref{eq:problem}.
	Then
	\begin{equation}
		\label{eq:reldistgenerale}
		-u(x)\ge\frac12\,\de_{\Si}(x)^2 \ \mbox{ for every } \ x\in\ol{\Om}.
	\end{equation}
	Moreover, if $\Om$ satisfies the strong $r_i$-uniform interior sphere condition relative to $B_+$, then it holds that
	\begin{equation}
		\label{eq:relationdist}
		-u(x) \ge \frac{r_i}{2}\,\de_{\Si} (x)  \ \mbox{ for every } \ x\in\ol{\Om}.
	\end{equation}
\end{lem}

\begin{proof}
	From Lemma \ref{lem:comparison in ball}, we know that $u \le 0$ in $\ol{\Om}$. 
	\par
	Fix $x \in \ol{\Om} \setminus \ol{\Si}$, let $r=\de_{\Si} (x)$, and consider the ball $B_r(x)$ with radius $r$ centered at $x$. 
	It is easy to check that the function defined by $w(y)=(|y-x|^2-r^2)/2$ for $y\in\ol{B}_r(x)$ satisfies
	\begin{equation*}
		\De w = N  \ \text{  in } \ B_r(x), \quad w =0 \ \text{  on } \ \pa B_r(x) ,
	\end{equation*}
	and
	$$
	w_\nu \ge w  \ \text{ on } T \cap B_r (x).
	$$
	The last inequality follows from the direct computations
	$$w= \frac{1+|x|^2 - r^2}{2} - \langle x , y \rangle
	\quad \text{ and } \quad 
	w_\nu = 1 - \langle x , y \rangle , \text{ for } y\in T \cap B_r (x)\subset \pa B ,
	$$ 
	and
	the trivial inequality $r^2 \ge 0 \ge |x|^2 -1$, which holds true for any $x \in \ol{B} \supset \ol{ \Om }$.
	
	If we choose $f= u- w$ and $\Om= B_r(x) \cap B_+$ in Lemma \ref{lem:comparison in ball}, we then have that $w \ge u$ in $\ol{B_r(x) \cap B_+}$ and hence, in particular, $-r^2/2 = w(x) \ge u(x)$. Thus, \eqref{eq:reldistgenerale} is proved.
	\par
	Next, assume that $\Om$ satisfies the strong $r_i$-uniform interior sphere condition relative to $B_+$. 
	If $\de_{\Si}(x) \ge r_i$, \eqref{eq:relationdist} immediately follows from \eqref{eq:reldistgenerale}. If $\de_{\Si } (x) < r_i$, instead, let $\ul{x}$ be the closest point in $\ol{\Si}$ to $x$ and call $B_{r_i}$ the relevant touching ball at $\ul{x} \in \ol{\Ga}_0$ (with radius $r_i$ and centered at the point $\ul{x} + r_i \frac{x- \ul{x}}{|x- \ul{x} |}$ that we denote by $x_0$), which contains $x$.
	Setting $w(y)=\left(|y - x_0|^2 - r_i^2 \right)/2$ and using that the center $x_0$ of the touching ball satisfies $|x_0|^2 \le 1 + r_i^2$, we infer that $w_\nu \ge  w$ on $T \cap B_{r_i}$. Hence, an application of Lemma \ref{lem:comparison in ball} with $f=u-w$ and $\Om= B_{r_i} \cap B_+$ gives that $w \ge u$ in $\ol{B_{r_i} \cap B_+}$. 
	As a consequence, being as $x \in B_{r_i} \cap \ol{B_+}$, we obtain that
	$$
	-u(x) \ge \frac{ r_i^2 - |x - x_0|^2 }{2}=
	\frac{( r_i + |x- x_0| )( r_i - |x- x_0| )}{2} \ge \frac{r_i}{2}\,(r_i -|x- x_0|) .
	$$
	This gives \eqref{eq:relationdist}, since $r_i - |x - x_0| = \de_{\Si } (x)$.
\end{proof}

\begin{rem}
	{\rm (i) Notice that the $r_i$-uniform interior sphere condition relative to $B_+$ guarantees the validity of the Hopf lemma for $u_\nu$ on $\Si$. The additional ``strong'' assumption is needed to obtain the Lipschitz growth of $u$ from $\Si$, i.e., \eqref{eq:relationdist}.
		
		(ii) In the classical setting (where $B_+$ is replaced by $\RR^N$), the improved growth in \eqref{eq:relationdist} remains true in the more general case in which $\Om$ satisfies an interior pseudoball condition (see \cite[Step 2 in the proof of Theorem I]{CPY} and \cite[Theorem~4.4]{ABMMZ}). In this regard, one may introduce a notion of pseudoball condition relative to $B_+$.
	}
\end{rem}


Let $u$ be the solution of \eqref{eq:problem}. We consider the harmonic function 
\begin{equation}\label{eq:def h^z}
h (x) = Q(x) - u(x), \ \ x \in \Om,
\end{equation}
where
\begin{equation}
\label{quadratic}
Q (x)=\frac12\, |x-z|^2.
\end{equation}
and $z\in\RR^N$ is some point to be chosen. 
As the next lemma shows, $h$ has to do with the numbers
\begin{equation}
\label{eq:rhoe e rhoi}
\rho_e= \max_{x \in \ol{\Si} }{|x-z|} \quad \rho_i=\min_{x \in \ol{\Si} }{|x-z|}.
\end{equation}
Note that we have:
$$
\ol{\Si} \subseteq \left( \ol{B}_{\rho_e}(z) \setminus B_{\rho_i}(z) \right) \cap \ol{B} .
$$

\begin{lem}
\label{lem:relation_osc_rhoe-rhoi e Hess h}
Fix $z\in\RR^N$. Then, we have that
\begin{equation*}
|\na^2 h|^2 = | \na^2 u|^2- \frac{(\De u)^2}{N} \ \text{ in } \ \Om .
\end{equation*}	
Moreover, it holds that 
\begin{equation}
\label{eq:rhoe-rhoi stima cone}
	\rho_e - \rho_i \le	\frac{8}{d_\Om} \left( \max_{ \ol{\Sigma} } h - \min_{ \ol{\Sigma}} h \right).
\end{equation}
\end{lem}

\begin{proof}
The identity easily follows by direct computation. 
\par
Next, let $x_1, x_2\in\ol{\Om}$ be such that $d_\Om=|x_1-x_2|$. It is clear that $x_1, x_2\in \ol{\Si} \cup T$. If $x_1, x_2\in \ol{\Si}$, then $d_\Om\le d_\Si$. If  $x_1, x_2\in \ol{T}$, then $d_\Om\le d_T\le d_{\ol{\Si} \cap \ol{T} }\le d_\Si$. In fact, the second inequality follows from the fact that $T$ is a spherical cap contained in a half sphere. If $x_1\in \ol{\Si}$ and $x_2\in T$ (or the other way around), we take $y\in \La = \ol{\Si}\cap \ol{T}$ and infer that
$$
d_\Om=|x_1-x_2|\le |x_1-y|+|y-x_2|\le d_\Si+d_T\le 2\,d_\Si.
$$
Thus, in any case we have that $d_\Om\le 2\,d_\Si$.
Now, if $x_1, x_2\in \ol{\Si}$ are such that $d_\Si=|x_1-x_2|$, we easily see that $d_\Si=|x_1-x_2|\le |x_1-z|+|z-x_2|\le 2\,\rho_e$, so that $d_\Om\le 4\,\rho_e$.
Using the last inequality together with
$$
\max_{ \ol{\Sigma}} h - \min_{ \ol{\Sigma} } h = \frac{1}{2}(\rho^2_e - \rho^2_i)=
\frac12\,(\rho_e+\rho_i)(\rho_e-\rho_i)\ge \frac12\,\rho_e\,(\rho_e-\rho_i) ,
$$
\eqref{eq:rhoe-rhoi stima cone} easily follows.
\end{proof}

\subsection{Special stability estimates}
\label{subsec:special-estimates}

In this section, we shall give the proof of Theorems \ref{thm:Improvement1}--\ref{thm:Improvement3}. To this aim, we must work on the fundamental identity \eqref{fundamental-identity}. We shall see that its right-hand side can be easily estimated in terms of the  deviation of $\nr u^2_\nu-R^2\nr_{1,\Si}$. 
Thanks to Lemma \ref{lem:relationdist}, the left-hand side of \eqref{fundamental-identity}, instead, can be bounded from below by the following weighted $L^2$-norm:
\begin{equation}\label{eq:weightdnorm_olal}
\nr \de_{\Si}^{ \tau } \na^2 h \nr_{2, \Om}.
\end{equation}
The appropriate exponent $\tau$ will be chosen as $\tau= 1$ in Theorems \ref{thm:Improvement1} and \ref{thm:Improvement2}, $\tau= 1/2$ in Theorem \ref{thm:Improvement3}, and $\tau= 3/2$ in Theorem \ref{thm:stability-general} below. 
The final stability estimates will then result from a bound of $\rho_e -\rho_i$ in terms of those relevant  weighted norms. This task will be achieved by means of Lemma \ref{lem:p>N + p<N and p=N}.
\par
Thus, we begin with the following lemma.

\begin{lem}[Weighted bounds for the Hessian matrix of $h$]
\label{lem:weighted-bounds}
Take $N\ge 2$. Let $\Om\subset \RR^N$ be a subdomain of $B_+ $ and define the number
\begin{equation}
m=\min\{ x_N: x\in\ol{\Om}\}.
\label{def-m}
\end{equation}
Let $u$ be the solution of \eqref{eq:problem} with Lipschitz constant $L$ be as in \eqref{eq:seminorm}. 
\par
For any choice of $z \in \RR^N$, let $h$ be the function defined in \eqref{eq:def h^z}.
The following statements hold true.

(i) If $\Om$ satisfies the $(\te, a)$-uniform interior cone condition, then we have that
 \begin{equation*}
 	\nr \de_{\Si}^{3/2} \na^2 h \nr_{2,\Om}^2 \le  (L+ 2) \,  \nr  u_\nu^2 - R^2 \nr_{1,\Si} 
\end{equation*}
and, if the number $m$ in \eqref{def-m} is positive,
\begin{equation*}
	\nr \de_{\Si} \na^2 h \nr_{2,\Om}^2 \le  \frac{L+2}{ m }\,  \nr  u_\nu^2 - R^2 \nr_{1,\Si}.
\end{equation*}

(ii) If $\Om$ satisfies the strong $r_i$-uniform interior sphere condition relative to $B_+$, then we have that
\begin{equation*}
	\nr \de_{\Si} \na^2 h \nr_{2,\Om}^2 \le \frac{L+2}{r_i}\, \nr  u_\nu^2 - R^2 \nr_{1,\Si}
\end{equation*}
and, if the number $m$ in \eqref{def-m} is positive,
\begin{equation*}
	\nr \de_{\Si}^{\frac{1}{2}} \na^2 h \nr_{2,\Om}^2 \le \frac{L+2}{ m\,r_i} \nr  u_\nu^2 - R^2 \nr_{1,\Si}.
\end{equation*}
%
\end{lem}
\begin{proof}
	In view of \eqref{eq:seminorm}, we have that $0<u_\nu \le L $ on $\Si$. Thus, being as $0\le x_N\le 1$ for  $x\in B_+$, we have that
	\begin{equation*}
		|u_\nu x_N - \lan X^q, \nu\ran| \le L+ 2   \quad \text{ on } \, \Si,
	\end{equation*}
	and hence, from \eqref{fundamental-identity} and Lemma \ref{lem:relation_osc_rhoe-rhoi e Hess h}, we get:
	\begin{equation}\label{eq:inequalitygrezzaperstability}
		\int_\Om x_N (-u) |\na^2 h|^2 dx \le \frac{L+2}{2}\, \nr  u_\nu^2 - R^2 \nr_{1,\Si}
	\end{equation}

(i)	Notice that 
	\begin{equation*}
		x_N \ge \de_{\Si} (x) \quad \text{ for any } x \in \Om \subset B_+ .
	\end{equation*}
By this inequality, \eqref{eq:inequalitygrezzaperstability}, and \eqref{eq:reldistgenerale}, the first desired inequality easily follows.
Also, the second desired inequality easily ensues by putting together \eqref{eq:inequalitygrezzaperstability}, \eqref{eq:reldistgenerale}, and the fact that $m>0$.

(ii) Since $\Om$ satisfies the strong $r_i$-uniform interior sphere condition relative to $B_+$, \eqref{eq:relationdist} holds true.
Thus, we have that
\begin{equation*}
	-x_N\,u(x) \ge \frac{r_i}{2}\, \de_{\Si} (x)^2.
\end{equation*}
This bound, together with \eqref{eq:inequalitygrezzaperstability} leads to the first desired inequality.
Next, by using \eqref{eq:relationdist} and the fact that $m>0$, we deduce that
\begin{equation*}
	-x_N\,u(x) \ge  \frac{m\,r_i}{2}\, \de_{\Si} (x).
\end{equation*}
Inserting this inequality into \eqref{eq:inequalitygrezzaperstability} gives the second desired inequality.
\end{proof}

Notice that, as already mentioned, the proofs of Theorems \ref{thm:Improvement1}--\ref{thm:Improvement3} will only entail the cases in this lemma where $0<\tau \le 1$. 
The desired conclusions will in fact be obtained by adapting to the present setting the arguments developed by the authors in \cite{Pog2, MP3, MP6}. The case where $\tau=3/2$ will instead be used for the proof of the more general result contained in Theorem \ref{thm:stability-general}, which requires a new and careful analysis.




The proofs of Theorems \ref{thm:Improvement1}-- \ref{thm:Improvement3} will result from 
Theorem \ref{thm:oldschoolapproachW22_al<1} below. In order to proceed, we recall from the introduction the convenient choice of $z$: 
\begin{equation*}
	z = \frac{1}{|\Om|} \left\{ \int_\Om x \, dx - \int_T u(x) \, x \, dx  \right\} .
\end{equation*}
Notice that, with this choice, we have that 
$$
\int_\Om \na h \, dx =\int_\Om (x-z) \, dx- \int_\Om \na u \, dx=  \int_\Om x \, dx - z\, |\Om| -\int_T u(x) \, x \, dx = 0 .
$$
This ensures that Corollary \ref{cor:JohnPoincareaigradienti} can be applied with $v=h$, $E=\Om$.


%
%
%
\begin{thm}
\label{thm:oldschoolapproachW22_al<1}
Let $N\ge 2$ and let $\Om \subset B_+$ be a domain satisfying the $(\te, a)$-uniform interior cone condition. 
\par
Let $u$ be solution of \eqref{eq:problem} and consider the function $h=Q-u$, where $Q$ is given by \eqref{quadratic} with $z$ as in \eqref{eq:choice z old school}.
Then, the following statements hold true.
\par
(i) There is some non-negative constant $c= c(N, \tau, \te , a )$ such that
\begin{equation}
\label{eq:N>_W22-oldapproac_al leq 1}
\rho_e-\rho_i\le c\,\begin{cases}
\nr \de_{\Ga}^{\tau} \, \na^2 h  \nr_{2,\Om}  &\mbox{ if } \ 0<\tau<2-\frac{N}{2}, \\
\nr\na h \nr_{ \infty , \Om}^{1-\ka_{N,\tau} } \nr \de_{\Ga}^{\tau} \,   \na^2 h  \nr_{2,\Om}^{\ka_{N,\tau}}  &\mbox{ if } \ 2-\frac{N}{2}<\tau\le 1, 
\end{cases} 
\end{equation}
where 
$$
\ka_{N,\tau}=\frac1{\tau+N/2-1}.
$$
\par
(ii) There is some non-negative constant $c= c(N, \te , a )$ such that
	\begin{equation}\label{eq:N=_W22-oldapproac_al leq 1}
		\rho_e-\rho_i\le c \, \nr \de_{\Ga}^\tau \, \na^2 h  \nr_{2,\Om} \, \max \left\{ \log \left( \frac{ \nr \na h \nr_{\infty , \Om} }{ \nr \de_{\Ga}^\tau \, \na^2 h  \nr_{2,\Om} } \right) , 1 \right\}
\end{equation}
with $\tau=2-N/2$.
\end{thm}

\begin{proof}
In both items, we will use at some point the trivial inequality
\begin{equation*}
\max_{ \ol{\Sigma}} h - \min_{ \ol{\Sigma}} h \le	\max_{\ol{\Om}} h - \min_{\ol{\Om}} h.
\end{equation*}
	(i) Let $2-N/2<\tau\le 1$.
	By using item (ii) of Lemma \ref{lem:p>N + p<N and p=N} with $E=\Om$, $f=h$, $p = N \ka_{N,\tau}  
	<N$, 
	and $q=\infty$, we find a constant $c=c(N, \tau, \te , a)$ such that
	\begin{equation*}
		\max_{\ol{\Om}} h - \min_{\ol{\Om}} h \le  c \,
		\nr\na h \nr_{ N \ka_{N,\tau} , \Om}^{\ka_{N,\tau}  } \nr\na h \nr_{ \infty , \Om}^{1-\ka_{N,\tau} } .
	\end{equation*}
By using \eqref{eq:rhoe-rhoi stima cone} and the trivial inequality,
we thus find a constant $c=c(N, \tau, \te , a)$ such that
$$
\rho_e - \rho_i \le c \, \nr\na h \nr_{ N \ka_{N,\tau} , \Om}^{\ka_{N,\tau}} \nr\na h \nr_{ \infty , \Om}^{1-\ka_{N,\tau} } .
$$
We point out that in \eqref{eq:rhoe-rhoi stima cone} $8/d_\Om$ can be replaced by $8/a$, since $\Om$ contains at least a cone of height $a$.

By applying Corollary \ref{cor:JohnPoincareaigradienti} with $E=\Om$, $f= h$, $r=N \ka_{N,\tau}$, $p=2$, $\al=\tau$, the second inequality in (i) easily follows. 
\vskip.1cm
Next, let $\tau<2-N/2$.
By the Sobolev immersion theorem (here, we are indeed applying item (i) of Lemma \ref{lem:p>N + p<N and p=N}  with $E=\Om$, $f=h$ and $p=N \ka_{N,\tau}>N$), we can find a constant $c=c(N, \tau, \te ,a)$ such that
	\begin{equation*}
	\max_{\ol{\Om}} h - \min_{\ol{\Om}} h \le  c \, \nr \na h \nr_{N \ka_{N,\tau} , \Om }.
	\end{equation*}
By again using \eqref{eq:rhoe-rhoi stima cone} and the trivial inequality,  we thus infer that
$$
\rho_e - \rho_i \le c \, \nr \na h \nr_{N \ka_{N,\tau} , \Om },
$$	
by possibly changing the relevant constant.
Hence, the first inequality in (i) follows by using Corollary \ref{cor:JohnPoincareaigradienti} with $E=\Om$, $f= h$, $r=N \ka_{N,\tau}$, $p=2$, $\al=\tau$.

	(ii)
	Let $\tau=2-N/2$. By using \eqref{eq:rhoe-rhoi stima cone}, the trivial inequality, and item (ii) of Lemma \ref{lem:p>N + p<N and p=N} with $E=\Om$, $f=h$, $p=N = 4 - 2 \tau$ and $q=\infty$, we find a constant $c=c(N, \te , a)$ such that
%
%
%
$$
\rho_e - \rho_i \le c \, \nr \na h \nr_{N, \Om } \max \left\lbrace \log\left( \frac{ \nr \na h \nr_{ \infty , \Om} }{\nr \na h \nr_{N, \Om} }\right) , 1 \right\rbrace .
$$
The inequality in (ii) then ensues by applying Corollary \ref{cor:JohnPoincareaigradienti} with $E=\Om$, $f= h$, $r=N$,
$p=2$, and $\al=\tau$.
\end{proof}

\begin{rem}[An explicit bound for $\nr \na h \nr_{\infty, \Om}$]
\label{rem:bound-grad-h}
	{\rm
	With the choice \eqref{eq:choice z old school}, we can easily obtain the following explicit bound:
\begin{equation*}
	\nr \na h \nr_{\infty, \Om} \le 2(L+1).
\end{equation*}
where $L$ is that defined in \eqref{eq:seminorm}.
\par
In fact, we have that
\begin{equation}\label{eq:aggiunta per chiarire dopo na h}
|\na h(x)|\le |\na u(x)|+|x-z|\le L+|x-z|  \quad \text{ for } x\in\ol{\Om} . 
\end{equation}
Moreover, we see that
\begin{multline*}
|x-z|\le\left|x-\frac1{|\Om|}\int_\Om y\,dy+\int_T u(y)\,y\,dS_y\right| \le \\
\frac1{|\Om|}\int_\Om |x-y|\,dy+\frac1{|\Om|}\int_\Om |\na u(y)|\,dy \le
d_\Om+L\le 2+L,
\end{multline*}
for $x\in\ol{\Om}$. In the second inequality, we used that
$$
\int_T u(y)\,y\,dS_y=\int_\Om \na u(y)\,dy,
$$
by the divergence theorem.

}
\end{rem}

%
%
%
%
%
%
%
%
%
%
%
%

We are now ready for the proofs of Theorems \ref{thm:Improvement1}, \ref{thm:Improvement2}, \ref{thm:Improvement3}.

\begin{proof}[Proof of Theorem \ref{thm:Improvement1}]
	The conclusion easily follows by combining the second inequality in item (i) of Lemma \ref{lem:weighted-bounds},
Theorem \ref{thm:oldschoolapproachW22_al<1} (with $\tau=1$), the trivial inequality
\begin{equation}\label{eq:trivialinequalitydistanceSigmaGamma}
	\de_{\Ga}(x) \le \de_{\Si}(x) \quad \text{ for } x \in \ol{\Om} ,
\end{equation}
and Remark \ref{rem:bound-grad-h}.
	%
\end{proof}

\begin{rem}\label{rem:John extension}
	{\rm 
	Taking into account Remark \ref{rem:0 sub-harmonic version}, Theorem \ref{thm:Improvement1} may be extended to the case where the uniform interior cone condition is dropped and replaced by weaker either cone-type or John-type conditions. 
}
\end{rem}

\begin{proof}[Proof of Theorem \ref{thm:Improvement2}]
	The desired estimate easily follows by combining the first inequality in item (ii) of Lemma \ref{lem:weighted-bounds}, Theorem \ref{thm:oldschoolapproachW22_al<1} (with $\tau=1$), \eqref{eq:trivialinequalitydistanceSigmaGamma}, and Remark~\ref{rem:bound-grad-h}.
	%
\end{proof}

\begin{proof}[Proof of Theorem \ref{thm:Improvement3}]
	%
	%
	The desired estimate easily follows by using the second inequality in item (ii) of Lemma \ref{lem:weighted-bounds}, Theorem \ref{thm:oldschoolapproachW22_al<1} (with $\tau=1/2$), \eqref{eq:trivialinequalitydistanceSigmaGamma}, and Remark~\ref{rem:bound-grad-h}. 
	%
\end{proof}

\subsection{The general stability estimate}\label{subsec:general stability}
In this section, we shall state and prove a stability estimate for general domains satisfying the $(\te, a)$-uniform interior cone condition. 
Compared to those proved in Section \ref{subsec:special-estimates}, in this case the stability rates are slightly poorer, as the following theorem shows.

\begin{thm}[General stability]
\label{thm:stability-general}
Set $N\ge 2$ and let $\Om$ be a domain contained in $B_+$ and satisfying
the $(\te,a)$-uniform interior cone condition. 
\par
Let $u\in W^{1,\infty}(\Om)\cap W^{2,2} (\Om)$ be the solution of \eqref{eq:problem} and assume that $L$ is a bound for $[u]_{C^{0,1}(\ol{\Om})}$.
Let $R$ be the number and point defined in \eqref{value-R-intro} and set
$$\rho(\Om) = \inf_{z \in \RR^N} \left( \rho_e -\rho_i \right) ,\quad \text{ with } \rho_e \text{ and } \rho_i \text{ as in \eqref{eq:rhoe e rhoi}.}$$ 
Then, the following estimates hold true.
\begin{enumerate}[(i)]	
\item
If $N \ge 3$, 
\begin{equation*}
\rho (\Om) \le c \, \nr u_\nu^2 - R^2 \nr_{1,\Si}^{1/(N+1)},
\end{equation*}
for some non-negative constant $c= c(N,\te, a, L)$.
\item
If $N = 2$, for any $0<\eta<1$, 
	\begin{equation*}
		\rho (\Om) \le c \, \nr u_\nu^2 - R^2 \nr_{1,\Si}^{1/(3 + 2\eta)}  ,
	\end{equation*}
for some non-negative constant $c= c(\te, a, L, \eta )$.
\end{enumerate}
\end{thm}

Notice that, in order to prove the special stability estimate in the previous section, we used Theorem  \ref{thm:oldschoolapproachW22_al<1}. Here, we stress that, since its proof is based on
Lemma \ref{lem:p>N + p<N and p=N} and Corollary \ref{cor:JohnPoincareaigradienti}, the relevant exponent $\tau$ had to be chosen in $[0,1]$. 
This fact allowed us to treat the cases of Lemma \ref{lem:weighted-bounds} with $\tau=1/2$ or $1$.
\par
However, if we want to treat the general case, we must choose $\tau=3/2$, as it is clear from Lemma \ref{lem:weighted-bounds}.  Thus, Theorem  \ref{thm:oldschoolapproachW22_al<1} is no longer useful and we must come up with another strategy.
The key idea is to obtain inequalities similar to those in Lemma \ref{lem:p>N + p<N and p=N}, but restricting the $L^p$-norms (appearing on the right-hand sides) to a suitable subset sufficiently far from the boundary. 

To this aim, for $\si \ge 0$, we define the \textit{parallel set}
\begin{equation}\label{eq:Omegasigma}
\Om_\si= \left\{x \in \Om \, : \, \de_{\Ga} (x) > \si \right\} ,
\end{equation}
where $\Ga$ denotes the boundary of $\Om$.
Being as $\Om$ a bounded domain (i.e., open and connected) satisfying the $(\te, a)$-uniform interior cone condition, by Lemma \ref{lem:Ruiz connected parallel} below, we know that there exists a positive constant $\de_0=\de_0 (\te, a , d_\Om)$ such that $\Om_\si$ is connected for any $0 \le \si \le \de_0$.
We now set 
\begin{equation}\label{eq:definition sigma 0}
\si_0= \min \left\{ \frac{a}{2} \frac{\sin \te}{ 1 + \sin \te } , \de_0 \right\} .
\end{equation}
Notice that for $0 \le \si \le \si_0 $, besides being connected, the domain $\Om_\si$ also satisfies the $(\te, a/2)$-uniform interior cone condition. The second assertion follows from Lemma \ref{lem:conoparallel} noting that $\si_0 \le \frac{a}{2} \frac{\sin \te}{ 1 + \sin \te } \le \frac{a}{4} $.
This ensures that Lemma \ref{lem:John-two-inequalities} and Corollary \ref{cor:JohnPoincareaigradienti} can be applied with $E=\Om_\si$.


The following lemma will be useful in the sequel.

\begin{lem}\label{lem:general v primo step per caso al >1}
Let $\Om  \subset \RR^N$, $N\ge2$, be a bounded domain satisfying the $(\te,a)$-uniform interior cone condition. Consider the parallel set $\Om_\si$ for  $0 < \si \le \si_0 $, where $\si_0$ is that given in \eqref{eq:definition sigma 0}.
\par
If $1<p<N$, we have that
\begin{equation*}
	\max_{\Ga} v - \min_{\Ga} v \le c  \, \left\lbrace \si^{1-\frac{N}{p}} \nr \na v \nr_{p, \Om_\si} + \left[ v \right]_{C^{0,1}(\ol{\Om})} \si \right\rbrace  ,
\end{equation*}
for any function $v \in C^{0,1}( \ol{ \Om } )$ subharmonic in $\Om_\si$ 
and some positive constant $c$ depending on $N, p, \te, a, d_\Om$.
\end{lem}

\begin{proof}
	Let $x_1$ and $x_2$ be points on $\Ga$ that respectively minimize and maximize $v$ on $\Ga$. For $j=1, 2$, define the point $y_j = x_j + \frac{2 \si }{\sin \te} \,\om_j $, where $\om_j$ is the axis of a cone $\cC_{j} \subset \Om$ with vertex at $x_j$, height $a$, and opening width $\te$. Since $\frac{2 \si  }{\sin \te} \le \frac{a}{ 1 + \sin \te}$ (being as $\si \le \si_0$), by trigonometry we have that the ball $B_{2 \si}(y_j)$ is contained in $\cC_{j}\subset\Om$. Hence, the ball $B_{\si}(y_j)$ is contained in $\Om_\si$.
\par
Now, the sub-harmonicity of $v$ gives that
\begin{multline*}
|v( y_j ) - v_{\Om_\si} | \le 
\frac{1}{|B|\, \si^N }\,\int_{ B_{\si}(y_j) }  |v - v_{\Om_\si}| \,dy \le \\
\frac{1}{\left(|B|\, \si^N\right)^{1/ q}} \, 
\left[\int_{ B_{\si}(y_j) }|v - v_{\Om_\si} |^q \,dy\right]^{1/q} \le 
\frac{1}{\left(|B|\, \si^N\right)^{1/ q}} \, 
\left[\int_{\Om_\si}|v - v_{\Om_\si} |^q \,dy\right]^{1/q} ,
\end{multline*}
for any $q > 1$, after an application of H\"older's inequality. Thus,
by the definition of $\left[v\right]_{C^{0,1} ( \ol{\Om}) }$, we can infer that
\begin{multline*}
| v(x_j) - v_{\Om_\si} | \le | v(y_j) - v_{\Om_\si} | + \frac{2 \si }{\sin \te} \, \left[v\right]_{C^{0,1} ( \ol{\Om}) }  \le \\	
c\, \left\{ \si^{-N/q} \left[\int_{\Om_\si}|v - v_{\Om_\si} |^q \,dy\right]^{1/q}+   \left[v\right]_{C^{0,1} ( \ol{\Om}) } \si   \right\} ,
\end{multline*}
for some constant $c=c(N, q,\te)$.
Therefore, by choosing $q=pN/(N-p)$ and applying \eqref{John-harmonic-poincare} with $E=\Om_\si$, $r=pN/(N-p)$, $p=p$, $\al = 0$, we conclude that
our desired inequality holds with an explicit constant $c=c( N , p, \te, a , d_\Om)$.
\end{proof}

%
%
%

\begin{cor}\label{cor:general v secondo step per caso al>1}
		Let $\Om$, $\si$, and $\Om_\si$ be as in Lemma \ref{lem:general v primo step per caso al >1} and take $\tau\ge 1$.
		For any subharmonic function of class $C^{0,1}( \ol{ \Om } )$ in $\Om_\si$ such that
		\begin{equation*}
		\int_{\Om_\si} \na v \, dx = 0,	
		\end{equation*}
we have the following.
\begin{enumerate}[(i)]
\item 
If $N\ge 3$, then		
\begin{equation*}
\max_{\Ga} v - \min_{\Ga} v \le c  \, \, \left\{ \si^{2-\frac{N}{2} -\tau} \nr \de_{\Ga}^\tau\, \na^2 v \nr_{2, \Om} + \left[ v \right]_{C^{0,1}(\ol{\Om})} \si \right\} ,
\end{equation*}
for some positive constant  $c=c(N, \tau, \te, a, d_\Om)$.
\item 
If $N=2$, then for any $0<\eta<1$ we have that
		\begin{equation*}
		\max_{\Ga} v - \min_{\Ga} v \le  c    \, \left\{ \si^{1- \eta -\tau} \nr \de_{\Ga}^\tau\, \na^2 v \nr_{2, \Om} + \left[ v \right]_{C^{0,1}(\ol{\Om})} \si \right\} ,	
		\end{equation*}
for some positive constant $c=c(N, \tau, \te, a, \eta , d_\Om )$.
	\end{enumerate}
\end{cor}
\begin{proof}
For convenience, we set $\Ga_\si=\pa\Om_\si$.
\par
(i) Let $N \ge 3$. By putting together Lemma \ref{lem:general v primo step per caso al >1} with $p=2$ and Corollary \ref{cor:JohnPoincareaigradienti} with $E=\Om_\si$, $f=v$, $r=2$, $p=2$, $\al=1$, we find that
\begin{multline*}
\max_{\Ga} v - \min_{\Ga} v \le 
c \, \left\{ \si^{1-\frac{N}{2}} \, \nr \de_{\Ga_\si} \, \na^2 v  \nr_{2, \Om_\si } + \left[ v \right]_{C^{0,1}(\ol{\Om})} \si \right\} \le 
\\
c \, \left\{ \si^{1-\frac{N}{2}} \, \nr \de_{\Ga} \, \na^2 v  \nr_{2, \Om_\si } + \left[ v \right]_{C^{0,1}(\ol{\Om})} \si \right\}
\end{multline*}
for some constant $c=c(N, \te, a, d_\Om)$,
being as $\de_{\Ga_\si} (x) \le \de_{\Ga} (x)$ for any $x\in \Om_\si$ .
We can now exploit our construction to further increase the exponent of the distance in the first summand at the right-hand side of the last inequality. 
\par
In fact, the definition \eqref{eq:Omegasigma} of $\Om_\si$ gives that
$$
\de_\Ga \le \si^{ 1 - \tau } \de_\Ga^\tau \ \mbox{ in } \ \Om_\si,
$$
and hence 
$$
\nr \de_{\Ga} \, \na^2 v  \nr_{2, \Om_\si }  \le 
\si^{ 1 - \tau}  \nr \de_{\Ga}^\tau \, \na^2 v  \nr_{2, \Om_\si } \le 
\si^{1- \tau}  \nr \de_{\Ga}^\tau \, \na^2 v  \nr_{2, \Om }.
$$
This is just what was left to prove.

(ii) Let $N=2$. By combining Lemma \ref{lem:general v primo step per caso al >1} with $p=2/(1+\eta)$, the H\"older inequality 
$$
\nr \na v \nr_{ 2 / ( 1 + \eta) , \Om_\si } \le  |\Om_\si|^{\eta/2} \, \nr \na v \nr_{2, \Om_\si} ,
$$
and Corollary \ref{cor:JohnPoincareaigradienti} with $E=\Om_\si$, $f=v$, $r=2$, $p=2$, $\al=1$, we find that
\begin{multline*}
\max_{\Ga} v - \min_{\Ga} v \le 
c  \, \left\{ \si^{ - \eta } \, \nr \de_{\Ga_\si} \, \na^2 v  \nr_{2, \Om_\si } + \left[ v \right]_{C^{0,1}(\ol{\Om})} \si \right\} \le \\
c  \, \left\{ \si^{ - \eta } \, \nr \de_\Ga \, \na^2 v  \nr_{2, \Om_\si } + \left[ v \right]_{C^{0,1}(\ol{\Om})} \si \right\}
\end{multline*}
for some positive constant $c=c(\te, a, \eta, d_\Om )$. Here, we also estimated the term $|\Om_\si|$ appearing in the H\"older inequality above by means of $|\Om_\si| \le |\Om|\le |B| \, d_\Om^N$.
The first summand at the right-hand side of the last inequality  can be estimated as in the proof of (i), and hence the desired result follows at once. 
\end{proof}

\begin{rem}\label{rem:diametro e volume upper bound gratis}
{\rm
%
%
If $\Om \subset B_+$, then the dependence on 
%
%
$d_\Om$ in the constants $c$ in Lemma~\ref{lem:general v primo step per caso al >1} and Corollary \ref{cor:general v secondo step per caso al>1} can be removed, being as 
%
%
$d_\Om \le 2$.	
}
\end{rem}

%
%
%

We are now ready to prove our general stability result.
We are going to prove the stability result for $\rho_e - \rho_i$ with the choice
\begin{equation}
	\label{eq:choice z caso general al > 1}
	z = \frac{1}{| \Om_{\si} |} \left\{\int_{ \Om_{\si} } x \, dx - \int_{ \Om_{\si} } \na u \, dx \right\},
\end{equation}
for a given value of $\si$, as specified below in the proof. The result in the statement of Theorem \ref{thm:stability-general} will follow noting that $\rho(\Om) \le \rho_e - \rho_i$.
With the choice of $z$ in \eqref{eq:choice z caso general al > 1}, the function $h$ defined in \eqref{eq:def h^z}-\eqref{quadratic} satisfies
$$
\int_{ \Om_{\si} } \na h \, dx =0,
$$
and hence, Corollary \ref{cor:general v secondo step per caso al>1} can be applied with $v=h$.
\begin{proof}[Proof of Theorem \ref{thm:stability-general}]
Let $\si_0= \si_0 ( \te , a )$ be that defined in \eqref{eq:definition sigma 0}, where the dependence on $d_\Om$ has been removed in light of Remark \ref{rem:diametro e volume upper bound gratis}.

(i) Combining item (i) of Corollary \ref{cor:general v secondo step per caso al>1} with $v=h$, $\tau=3/2$ and the trivial inequality
\begin{equation}\label{eq:oscillation altra trivial diversa}
\max_{\ol{\Si}} h -\min_{\ol{\Si}} h \le \max_{\Ga} h -\min_{\Ga} h
\end{equation}
gives that
\begin{equation}
\label{eq:step1 final proof_perchiarezza}
\max_{ \ol{\Si} } h - \min_{ \ol{\Si} } h \le  c \left\{ \si^{ -\frac{N-1}{2} } \nr \de_{\Ga}^{3/2} \, \na^2 h  \nr_{2, \Om } + [h]_{C^{0,1}(\ol{\Om})} \, \si \right\},
\end{equation}
for any $0<\si \le \si_0$. By Remark \ref{rem:diametro e volume upper bound gratis}, here $c=c(N, \te , a)$.
\par
Now, the term $[h]_{C^{0,1}(\ol{\Om})}$ can be bounded
by recalling
\eqref{eq:aggiunta per chiarire dopo na h} and using that, by~\eqref{eq:choice z caso general al > 1},
\begin{equation*}
	|x-z| \le \frac1{|\Om_\si|}\int_{\Om_\si} |x-y|\,dy+\frac1{|\Om_\si|}\int_{\Om_\si} |\na u(y)|\,dy \le d_\Om +L \le 2+L ,
\end{equation*}
being as $\Om_\si\subset\Om\subset B_+$. 
As a consequence, we get the bound:
\begin{equation}
\label{eq:Gvalue_ga=1_perchiarezza}
[h]_{C^{0,1}(\ol{\Om})}\le 2\,(L+1).
\end{equation}
%
%
Putting together \eqref{eq:Gvalue_ga=1_perchiarezza}, \eqref{eq:step1 final proof_perchiarezza}, \eqref{eq:trivialinequalitydistanceSigmaGamma}, and the first inequality in item (i) of Lemma~\ref{lem:weighted-bounds} gives that
\begin{equation}
\label{eq:step 2 final proof}
\max_{ \ol{\Si} } h - \min_{ \ol{\Si} } h \le  2 \, c \, (L+1)  \left\{ \si^{ -\frac{N-1}{2} }\nr u_\nu^2 - R^2  \nr_{2, \Om }^{1/2} +  \si \right\}.
\end{equation}

We now fix
$$
\si = \min \left\{\nr u_\nu^2 - R^2  \nr_{2, \Om }^{1/(N+1)} , \si_0 \right\},
$$
so as to minimize in $\si  \in (0, \si_0]$ the right-hand-side of \eqref{eq:step 2 final proof}.
We then distinguish two cases.
\par
If $\nr u_\nu^2 - R^2  \nr_{2, \Om }^{1/(N+1)} < \si_0$, we have that $\si = \nr u_\nu^2 - R^2  \nr_{2, \Om }^{1/(N+1)}$, and hence \eqref{eq:step 2 final proof} becomes
\begin{equation}\label{eq:step 3 final proof}
\max_{ \ol{\Si} } h - \min_{ \ol{\Si} } h \le 4 c  \, (L+1) \, \nr u_\nu^2 - R^2  \nr_{2, \Om }^{1/(N+1)} .
\end{equation}
%
%
Otherwise, we easily obtain that
\begin{multline*} 
\max_{ \ol{\Si} } h - \min_{ \ol{\Si} } h  \le [h]_{C^{0,1}(\ol{\Om})}  \, d_{\Si} \le [h]_{C^{0,1}(\ol{\Om})}  \, d_{\Om} \le \\
4 (L+1)
\le 4 \, \si_0^{-1} (L+1 ) \,  \nr u_\nu^2 - R^2  \nr_{2, \Om }^{1/(N+1)} , 
\end{multline*}
where, in the third inequality, we used \eqref{eq:Gvalue_ga=1_perchiarezza} 
and that $d_\Om \le 2$.
Thus, \eqref{eq:step 3 final proof} always holds for some constant $c=c(N,\te,a)$.
The desired conclusion, then easily follows by recalling \eqref{eq:rhoe-rhoi stima cone}.

(ii) Fix $0<\eta<1$.
Combining item (ii) of Corollary \ref{cor:general v secondo step per caso al>1} with $v=h$, $\tau=3/2$, and \eqref{eq:oscillation altra trivial diversa} gives that
\begin{equation*}
\max_{ \ol{\Si} } h - \min_{ \ol{\Si} } h \le c  \left\{ \si^{-\eta-1/2 } \nr \de_{\Ga}^{3/2} \, \na^2 h  \nr_{2, \Om }  + [h]_{C^{0,1}(\ol{\Om})} \, \si  \right\},
\end{equation*}
for any $0<\si \le \si_0$. 
Putting together the last inequality, \eqref{eq:Gvalue_ga=1_perchiarezza}, \eqref{eq:trivialinequalitydistanceSigmaGamma}, and the first inequality in item (i) of Lemma \ref{lem:weighted-bounds}, we infer:  
$$
\max_{ \ol{\Si} } h - \min_{ \ol{ \Si} } h \le 2c\,(L+1)  \left\{ \si^{-\eta-1/2 } \nr u_\nu^2 - R^2  \nr_{2, \Om }^{1/2}  +  \si  \right\}.
$$
\par
We now fix
$$
\si = \min \left\{ \nr u_\nu^2 - R^2  \nr_{2, \Om }^{1/(3 + 2 \eta)} , \si_0  \right\},
$$
so as to minimize in $\si \in (0, \si_0]$ the right-hand-side, and conclude by the same analysis performed in item (i).
\end{proof}

\appendix

\section{Remarks on the uniform cone condition}\label{appendix:cone and John}

In this appendix, we detail some geometrical facts and amend an inaccuracy contained in \cite{MP6}.

\subsection{Some geometrical facts}
As already mentioned, the uniform $(\theta, a )$-interior cone condition adopted in the present paper is equivalent to 
%
%
the strong local Lipschitz property of Adams \cite[p. 66]{Ad} and to the uniform Lipschitz regularity in \cite[Section III]{Ch} and \cite[Definition 2.1]{Ru}.
By putting together \cite[Proposition 4.1 in the Appendix]{Ru} and \cite[Proposition III.1]{Ch}, we easily infer the following result.

\begin{lem}
\label{lem:Ruiz connected parallel}
Let $\Om$ be a bounded domain
satisfying the uniform $(\te, a )$-interior cone condition.
	There exists a positive constant $\de_0$ depending on $a, \te$, and $d_\Om$ such that, for any $\si \le \de_0$, the parallel set $\Om_\si=\left\{ x\in \Om  \, : \, \de_{\Ga }(x) > \si \right\}$ is connected.
\end{lem} 

A domain $\Om$ in $\RR^N$ is a {\it $b$-John domain}, with $b \ge 1$, if each pair of distinct points $x_1$ and $x_2$ in $\Om$ can be joined by a curve $\psi: [0,1] \to \Om$ such that
$\psi(0)=x_1$, $\psi(1)=x_2$, and
\begin{equation*}
	b\,\de_{\Ga} (\psi(t)) \ge \min\left\{ |\psi(t) - x_1|, |\psi(t) - x_2| \right\}.
\end{equation*}
A curve satisfying the previous inequality is called a {\it John curve}.
By using the previous lemma, we now prove that domains satisfying the uniform $(\theta, a )$-interior cone condition are $b$-John domains and provide an explicit estimate for $b$ in terms of $\te, a, d_\Om$.

\begin{lem}\label{lem:cone condition strong implies John}
	Let $\Om$ be a bounded domain satisfying the uniform $(\te, a )$-interior cone condition. Then, $\Om$ is a $b$-John domain with
	$$
	b \le \max \left\{ \frac{1}{\sin(\te)} , \frac{d_\Om}{\min \left\{ \frac{a}{2} \frac{\sin \te}{1+ \sin \te} , \, \de_0 \right\} } \right\} ,
	$$
	where $\de_0$ is the constant appearing in Lemma \ref{lem:Ruiz connected parallel}.
\end{lem}

\begin{proof}
	Set $\si= \min \left\{ \frac{a}{2} \frac{\sin \te}{1+ \sin \te} , \, \de_0 \right\}$. Lemma \ref{lem:Ruiz connected parallel} guarantees that any two points $x_1, x_2\in\Om_\si$ can be joined by a curve $\psi: [0,1]\to\Om_\si$. Also, we easily compute that 
	$$
	\frac{\min\left\{ |\psi(t) - x_1|, |\psi(t) - x_2| \right\}}{\de_{\Ga} (\psi(t))} \le 
	\frac{d_\Om}{\de_{\Ga} (\psi(t))} \le \frac{ d_\Om }{\min \left\{ \frac{a}{2} \frac{\sin \te}{1+ \sin \te} , \, \de_0 \right\} } . 
	$$
	
	On the other hand, if $x_j$ (for $j=1$ and/or $2$) is a point in $\Om\setminus \Om_\si$, then we can find a point $y_j \in \Om_\si$ and another curve $\phi_j$, joining $x_j$ to $y_j$,  such that
\begin{equation*}
\label{eq:stepJohnstimaesplicita}
		\frac{\min\left\{ |\phi_j (t) - x_1|, |\phi_j (t) - x_2| \right\}}{\de_{\Ga} (\phi_j (t))} \le
\frac{1}{\sin \te} .
\end{equation*}
In fact, we have that
	$\de_{\Ga} (x_j) \le \si \le \frac{a}{2} \frac{\sin \te}{1+ \sin \te} \le a/4 $. Hence, if $x^j$ is the projection of $x_j$ on $\Ga$, \eqref{def:cone strong} gives that $x_j + \cC_{\om} \subset \Om$ with $\om=\om_{x^j}$. If we set $y_j = x_j + \frac{a}{1+ \sin \te} \om$ (which is a point on the axis of the cone $x_j + \cC_{\om}$), by some trigonometry we have that $\de_{\Ga} (y_j) \ge \de_{\pa (x_j + \cC_\om)}(y_j) = a \frac{\sin \te}{ 1 +\sin \te} > \frac{a}{2} \frac{\sin \te}{ 1 +\sin \te}$. In particular, $y_j \in \Om_\si$. 
\par
For $\ell>0$, the choice
	\begin{equation*}
		\phi_j(t)=
		\begin{cases}
			x_1 + \frac{t}{\ell} (y_1 - x_1) \quad & \text{if } j=1 , 
			\\
			y_2 + \frac{t}{\ell} (x_2 - y_2) \quad & \text{if } j=2 ,
		\end{cases}
		\quad \quad
		t \in [0, \ell],
	\end{equation*}
is clearly admissible. Moreover, for any $x_1, x_2 \in \Om$, it allows to create a suitable curve from $x_1$ to $x_2$  by joining together $\phi_1$ (if $x_1 \in \Om\setminus \Om_\si$), a curve contained in $\Om_\si$, and $\phi_2$ (if $x_2 \in \Om\setminus \Om_\si$).

	In any case, for any $x_1, x_2 \in \Om$ we can always find a John curve $\psi$ from $x_1$ to $x_2$ such that
	\begin{equation*}
		\frac{\min\left\{ |\psi(t) - x_1|, |\psi(t) - x_2| \right\}}{\de_{\Ga} (\psi(t))} 
		\le \max \left\{ \frac{1}{\sin(\te)} , \frac{d_\Om}{\min \left\{ \frac{a}{2} \frac{\sin \te}{1+ \sin \te} , \, \de_0 \right\} } \right\} ,
	\end{equation*}
	and the conclusion follows.
\end{proof}

We now prove the following useful result.
\begin{lem}
\label{lem:conoparallel}
	Let $\Om$ satisfy the $(\theta, a )$-uniform interior cone condition. Then, the parallel set $\Om_\si=\left\{ x\in \Om  \, : \, \de_{\Ga}(x) > \si \right\}$ satisfies the $(\theta, a/2 )$-uniform interior cone condition, for any $\si\le a/4$.
\end{lem}
\begin{proof}
	Let $x$ be any point on $\pa \Om_\si$ and let $y $ be a point in $\Ga$ (not necessarily unique) such that $\de_{\Ga}(A)=|x-y| = \si$. Since $\Om$ satisfies the $(\te, a )$-uniform interior cone condition, we set $\cC_{\om}$ to be a cone satisfying \eqref{def:cone strong} (with $x=y$).
	Since $B_\si (x)\subset\Om$, by using \eqref{def:cone strong} we can easily verify that $x+\cC_{\om}\cap B_{a/2} \subset \Om_\si$.
\par
Moreover, we can also check that
	\begin{equation*}
\label{eq:stepconoparallel}
		w +\cC_{\om}\cap B_{a/2} \subset \Om_\si \ \text{ for every } \ w \in B_{a/2} (x) \cap \ol{\Om}_\si .
	\end{equation*}
Since $x$ is chosen arbitrarily in $\pa \Om_\si$, the last inclusion gives that $\Om_\si$ satisfies the $(\te, a/2 )$-uniform interior cone condition.
The last inclusion holds by noting that,
for any $w\in B_{a/2}(x) \cap \ol{\Om}_\si$, we have that $B_\si(w)\subset\Om$ (by definition of $\Om_\si$) and $B_\si(w) \subset B_a (y)$ (being as $\si \le a/4$). Hence, we can argue as above to get that $w +\cC_{\om}\cap B_{a/2} \subset \Om_\si$.
\end{proof}

\begin{figure}[h]
\begin{center}
\begin{tikzpicture}[scale=.5] 

\draw plot [thick, mark=*, smooth] coordinates {(1.7,1)}; 
\draw plot [thick, mark=*, smooth] coordinates {(1,1)}; 

\draw (1,1) circle (40mm); 

\fill [gray!30] (1.7,1) -- (1.7,3) -- (3.5,2); 
\fill [gray!30, draw]
(1.7,1) ++(30:20mm) arc (30:90:20mm);  
\draw [thick] (1.7,1) -- (1.7,3); 
\draw [solid] (1,1) -- (-3,1); 
\draw [solid] (1.7,1) -- (.2,-.3); 
\draw [thick] (1.7,1) -- (3.45,2); 
\draw [solid, densely dashed] (1,1) -- (1,5.5); 
\draw [solid, densely dashed] (2,.35) -- (5.5,2.35); 
\draw [thick] (1.7,1) circle (20mm); 
\draw [solid] (1.7,1) circle (7mm); 

\node at (1.7,.6) {$x$}; 
\node at (.7,.65) {$y$}; 
\node at (-1.2,1.2) {$a$}; 
\node at (1,-.3) {$a/2$}; 
\end{tikzpicture}
\end{center}
\caption{The construction of Lemma \ref{lem:conoparallel}. Here, $x\in\pa\Om_\si$ and  $y\in\Ga=\pa\Om$ is such that $|x-y|=\de_\Ga(x)=\si\le a/4$. 
The shaded region is the cone $x+\cC_{\om}\cap B_{a/2}$. By \eqref{def:cone strong}, the region bounded by the dashed lines and containing the smallest disk is contained $\Om$. }
\end{figure}

%
%
%
%

\subsection{Errata corrige of \texorpdfstring{\cite[Corollary 2.3 and Theorems 2.4 and 2.7]{MP6}}{[17, Corollary 2.3 and Theorems 2.4 and 2.7]}}
\label{sec:errata corrige}
In \cite{MP6}, we assumed the following notion of cone condition, which is strictly weaker than the one adopted in the present paper. A bounded domain $\Om \subset \RR^N$ with boundary $\Ga$ satisfies the $(\te, a)$-uniform interior cone condition if, for every $x \in \ol{\Om}$, there is a cone $\cC_x$ with vertex at $x$, opening width $\te$, and height $a$, such that
$\cC_{x} \subset\Om$ and  $\ol{\cC}_{x} \cap \Ga = \{ x \} $, whenever $x \in \Ga$.
We will refer to this definition as the \textit{old cone condition}. It is easy to check that this condition is verified (with same $\te$ and $a$), if $\Om$ satisfies the (new) $(\te,a)$-uniform interior cone condition adopted in Section \ref{sec:preliminary-estimates}.

It is a classical result (\cite{Ad, Ru}) that if $\Om$ is a bounded domain satisfying the old cone condition, then there exists a positive constant $C_p(\Om)$ ---  the $(p,p)$-Poincar\'e constant --- such that 
\begin{equation*}\label{eq:errata corrige pp Poincare}
	\nr f - f_\Om\nr_{p,\Om}\le C_p(\Om) \nr\na f\nr_{p,\Om} \ \text{ for any } \ f \in W^{1,p}(\Om).
\end{equation*}

\par
We realized that the proof of \cite[Corollary 2.3]{MP6} contains a mistake.
Here, we correct that proof. 
The amended proof below shows that the constant $c$ in \cite[Corollary 2.3]{MP6} depends not only on $N$, $p$, $\te$, $a$, \textit{but also} on $C_p(\Om)$. As a consequence, the dependence on $C_p(\Om)$ should be added also in the constants $c$ of \cite[Theorems 2.4 and 2.7]{MP6}. 
Since, when $\Om$ is of class $C^2$, $C_p(\Om)$ can be estimated in terms of the radius $r_i$ of the uniform interior sphere condition and the diameter $d_\Om$ (see \cite[item (iii) of Remark 2.4]{MP3}), \cite[Lemma 3.2]{MP6} remains true with a constant $c=c(N, p , r_i, d_\Om)$ and the rest of the paper remains unchanged. 

\begin{proof}[Amended proof of {\cite[Corollary 2.3]{MP6}}]
	By using \cite[(2.3)]{MP6}, we have that 
$$
| f(x) - f_{\cC_x}|
\le c_{N,p}\,a
\left( \frac{1}{|\cC_x|}  \int_{\cC_x} | \na f |^p \, dx \right)^{1/p}
\le c_{N,p}\,\frac{a\ }{|\cC_x|^{1/p}} \nr\na f\nr_{p,\Om}.
$$
(Note that in \cite{MP6}, differently from the present paper, the $L^p$ norms were normalized by the Lebesgue measure of the domain.)
\par
Next, we easily infer that	
\begin{multline*}
|f_{\cC_x} - f_\Om| \le \frac{1}{|\cC_x|} \int_{\cC_x} | f - f_\Om | \, dx
\le \frac{1}{|\cC_x|^{1/p}} \left( \int_{\cC_x} | f - f_\Om |^p \, dx \right)^{1/p} \le \\
\frac{1}{|\cC_x|^{1/p}} \nr f - f_\Om\nr_{p,\Om} 
\le \frac{C_p(\Om)}{|\cC_x|^{1/p}}\, \nr \na f\nr_{p,\Om}.
\end{multline*}
All in all, we conclude that
$$
| f(x) - f_{\Om}|\le | f(x) - f_{\cC_x}|+|f_{\cC_x} - f_\Om|  \le c \, \nr \na f\nr_{p,\Om},
$$
for some constant $c$ that depends on $N, p, \te, a$, and $C_p(\Om)$.     
\end{proof}
 
\begin{rem}
 {\rm
As pointed out in Remark \ref{rem:on the proof of John and Poincare},
%
%
if $\Om$ is a bounded $b$-John domain, $C_p(\Om)$ can be estimated in terms of $b$ and $d_\Om$. In turn, if $\Om$ satisfies the new cone condition of the present paper, the John parameter $b$, and hence $C_p(\Om)$, can be estimated in terms of the parameters $\te$, $a$ of the relevant definition, and $d_\Om$. From this observation, the statement of Lemma \ref{lem:p>N + p<N and p=N} easily follows.
\par
On the contrary, the old cone condition adopted in \cite{MP6} is not sufficient to give an estimate of the $(p,p)$-Poincar\'e constant (see, e.g., \cite{Ru}), and hence neither of the John parameter. 
In fact, reasoning as in \cite[Example 2.6]{Ru}, one can construct a family of (uniformly) bounded domains $\Om^\ve$ sharing the same (fixed) parameters of the old cone condition and a sequence $u_\ve \in W^{1,2}(\Om^\ve)$ such that
$$
\int_{\Om^\ve} u_\ve \, dx = 0, \quad \int_{\Om^\ve}|\na u_\ve|^2 \, dx \to 0,
$$
while
$\int_{\Om^\ve} u_\ve^2 \, dx$ remains bounded away from zero.
}
\end{rem}

\section*{Acknowledgements}
R. Magnanini is partially supported by the Gruppo Nazionale Analisi Mate\-matica Probabilit\`a e Applicazioni (GNAMPA) of the Istituto Nazionale di Alta Matema\-ti\-ca (INdAM).

G. Poggesi
is supported by the Australian Research Council (ARC) Discovery Early Career Researcher Award (DECRA) DE230100954 ``Partial Differential Equations: geometric aspects and applications'' and the 2023 J G Russell Award from the Australian Academy of Science, and is member of the Australian Mathematical Society (AustMS) and the Gruppo Nazionale Analisi Matematica Probabilit\`{a} e Applicazioni (GNAMPA) of the Istituto Nazionale di Alta Matematica (INdAM).

The authors are grateful to the referee, whose comments helped to improve the manuscript.

\subsection*{Data availability statement} Data sharing not applicable to this article as no datasets were generated or analysed during the current study.

\end{document}